\documentclass{article}

\usepackage[english]{babel}
\usepackage[a4paper]{geometry}
\usepackage{amssymb,amsmath,amsthm}
\usepackage{graphicx,cite,xcolor}
\usepackage{longtable,pdflscape,booktabs,caption,multicol}

\usepackage[hyphens]{url}
\usepackage[colorlinks=true,citecolor=black,linkcolor=black,urlcolor=blue]{hyperref}

\usepackage{listings}
\definecolor{mauve}{rgb}{0.58,0,0.82}
\definecolor{dkgreen}{rgb}{0,0.6,0}
\lstdefinestyle{pitonche} {
    language = Python,
    basicstyle = footnotesizettfamily,
    showspaces = false,
    showstringspaces = false,
    breakautoindent = true,
    flexiblecolumns = true,
    keepspaces = true,
    stepnumber = 1,
    xleftmargin = 0pt
}
\usepackage{enumitem}
\usepackage{mathtools}

\theoremstyle{plain}
\newtheorem{theorem}{Theorem}[section]
\newtheorem{lemma}[theorem]{Lemma}
\newtheorem{corollary}[theorem]{Corollary}
\newtheorem{proposition}[theorem]{Proposition}

\newtheorem{problem}[theorem]{Problem}

\DeclareMathOperator{\Tr}{Tr}

\theoremstyle{definition}

\theoremstyle{remark}
\newtheorem*{remark}{Remark}

\newcommand{\doi}[1]{\url{https://doi.org/#1}}

\usepackage{float}
\usepackage{tikz}
\usepackage{subfig}
\usetikzlibrary{automata}
\tikzset{
    edge/.style={-{Latex[scale=1.2]}},
}

\usepackage{blkarray, bigstrut}
\usepackage{bm}

\def\dj{d\kern-0.4em\char"16\kern-0.1em}
\def\Dj{\hbox{\raise0.3ex\hbox{-}\kern-0.4em  D}}

\usepackage{authblk}

\author[1,2]{Ivan Damnjanovi\'{c}}
\author[3]{Anran Xu}
\author[4,$\dag$]{Kexiang Xu}

\affil[1]{FAMNIT, University of Primorska, Koper, Slovenia}
\affil[2]{Faculty of Electronic Engineering, University of Ni\v{s}, Ni\v{s}, Serbia}
\affil[3]{School of Mathematical Sciences, Nanjing Normal University, Nanjing, China}
\affil[4]{School of Mathematics, Nanjing University of Aeronautics \& Astronautics, Nanjing, China}

\title{On the transmission irregular trees\\with the maximum Wiener index\footnote{Email addresses:\ \texttt{ivan.damnjanovic@famnit.upr.si} (I.\ Damnjanovi\'{c}), \texttt{sdf000112@126.com} (A.\ Xu), \texttt{kexxu1221@126.com} (K.\ Xu). $\dag$ Corresponding author.}}

\date{}

\begin{document}

\maketitle

\begin{abstract}
The transmission of a vertex $v$ in a (chemical) graph $G$ is the sum of distances from $v$ to  other vertices in $G$. If any  two vertices of $G$ have different transmissions, then $G$ is transmission irregular. The Wiener index $W(G)$ of a graph $G$ is the sum of all distances between all unordered pairs of vertices in $G$, which has another formula as the half of the sum of transmissions of all vertices of $G$.
In this paper, we consider the Wiener index maximization problem on the set of transmission irregular trees of a given order $n \in \mathbb{N}$. We solve the problem for all odd values of $n$ and for almost all even values of $n$. Each resolved extremal problem has a unique solution that is a chemical tree.
\end{abstract}
\par\vspace{2mm}
\noindent{{\bf Keywords:}  tree; transmission irregular graph; Wiener index }
\par\vspace{1mm}

\noindent{{\bf AMS Classification (2020):}  05C05, 05C12, 05C35}

\section{Introduction}

\ \indent Throughout this paper we only consider simple, undirected and connected graphs. Let $G=(V(G),E(G))$ be a graph with vertex set $V(G)$ and edge set $E(G)$. For any graph $G$, let $|G|$ denote the order of $G$. We will use $d_G(v)$ to denote the \textit{degree} of a vertex $v$ in graph $G$ (or just $d(v)$ if the graph $G$ is clear from the context). As usual, for two distinct vertices $u,v\in V(G)$, let $d_G(u,v)$ be the \textit{distance} between $u$ and $v$ in a graph $G$, that is, the number of edges in a shortest path connecting $u$ and $v$ in $G$. A connected graph with maximum degree at most four is called a \textit{chemical graph}\cite{Tr1992}, which represents a chemical molecule in Mathematical Chemistry for the reason that any carbon atom in organic compounds forms four chemical bonds. Any other undefined notations and terminology on graph theory can be found in \cite{bondy}.

 As an oldest and well-studied distance-based topological index, the \textit{Wiener index} $W(G)$, introduced by H. Wiener \cite{Wiener1947} when calculating the boiling point of alkanes, of a graph $G$ is defined as the sum of all distances of all unordered pairs of vertices in $G$. The \textit{transmission}, denoted by $\Tr_G(v)$ (or just $\Tr(v)$ if the graph $G$ is clear from the context), of vertex $v$ in a graph $G$ is the sum of all distances from $v$ to all other vertices in $G$, that is, $\Tr_G(v)=\sum\limits_{u\in V(G)\setminus\{v\}}d_G(v,u)$. From the respective definitions of Wiener index and transmission, it is clear that  \begin{equation}\label{intro_aux_1}
        W(G)=\frac{1}{2}\sum\limits_{v\in V(G)}\Tr_G(v)
    \end{equation}
    for any graph $G$. Some recent mathematical results on the Wiener index of graphs are reported in \cite{DAS21,LW24,KR23,XDGW22}. Determining the extremal graphs in a given class of graphs with respect to the Wiener index is attractive for many mathematical researchers. In addition to the survey paper \cite{XD2014}, please search in \cite{KST} for recently related results on this topic.

For a graph $G$, its \textit{transmission set} is $\Tr(G)=\{\Tr_G(v) \mid v\in V(G)\}$. If $|\Tr(G)|=|G|$ with $|G|>1$, that is, all vertices in $G$ have different transmissions, then $G$ is \textit{transmission irregular} \cite{AlKl2018}, or TI for short. In \cite{AlKl2018} it was pointed out that almost all graphs are not transmission irregular, that is to say, the transmission irregular graphs are rare in the set of all connected graphs. Based on this point, characterizing the TI graphs in a given class of graphs is an interesting task. Some relevant results on the various TI graphs are published in \cite{ASJ2022,ASA2022,damn-2024,dams2024,dobrynin-2019b,dobrynin-2019c,XK21,XTK23}.

Based on the expression \eqref{intro_aux_1} of the Wiener index with transmissions, a natural problem occurs into our mind as follows:

\textit{What are the extremal graphs among a given class of TI graphs with respect to the Wiener index?}

In this paper we focus on the problem of determining the maximum TI trees among all TI trees of a given order with respect to the Wiener index. The paper is organized as follows. In Section \ref{sc_prel} some preliminary results are listed or proved for the use in the subsequent proofs. The maximum TI trees of all odd orders and almost all even orders are characterized, respectively, in Sections \ref{sc_odd} and \ref{sc_even}. Some relevant open problems are proposed in Section \ref{sc_conclusion}.

\section{Preliminaries}\label{sc_prel}

\ \indent For any positive integer $k$, we denote by $[k]$ the set $\{1, 2, \ldots, k\}$. We denote by $\mathcal{PS}$ the set of all perfect square numbers. For a set $A$ of numbers with a real number $t$, the coset $A+t$ is just $\{a + t \mid a\in A\}$.

For any $k \ge 3$, let $S(a_1, a_2, \ldots, a_k)$ denote the \textit{starlike tree} with $k$ pendent paths of respective lengths $a_1, a_2, \ldots, a_k$. For any $n \ge 3$ and $k \ge 2$, let $C_n(a_1, a_2, \ldots, a_k)$ be the \textit{ordinary caterpillar tree} of order $n + k$ that arises from the path $P_n = v_1 v_2 \cdots v_n$ by attaching a leaf to each of the vertices $v_{a_1}, v_{a_2}, \ldots, v_{a_k}$, with $2 \le a_i \le n - 1$ for $i \in [k]$. Also, for any $n \ge 3$, let $C_n(a_1, b_1; a_2, b_2; \ldots; a_k, b_k)$, with $k \ge 2$ and $\max\limits_{i \in [k]} b_i \ge 2$, be the\textit{ variant caterpillar tree} of order $n + \sum\limits_{i = 1}^k b_i$ that arises from the path $P_n = v_1 v_2 \cdots v_n$ by attaching a pendent path of length $b_i$ to the vertex $v_{a_i}$, where $2 \le a_i \le n - 1$, for each $i \in [k]$.

For a graph $G$ with $uv\in E(G)$, we denote by $n_G(u, v)$  the number of vertices in $G$ that are closer to $u$ than to $v$ (or just $n(u, v)$ if the graph $G$ is clear from the context). A relation between $n_G(u,v)$ and transmissions is presented in the following result.

\begin{lemma}[\hspace{1sp}{\cite{Bala,EJS1976}}]\label{neighbors_lemma}
    For any connected graph $G$ with $uv\in E(G)$, we have
    \[
        \Tr_G(u) - \Tr_G(v) = n_G(v, u) - n_G(u, v).
    \]
\end{lemma}
 In the following we list two results on the Wiener index of graphs.
\begin{lemma}[\hspace{1sp}{\cite{BSI2008}}]\label{wiener_fusion_lemma}
    Let $G_1$ and $G_2$ be two connected graphs with  $v_i \in V(G_i)$ for $i \in [2]$. Let $G$ be the graph of order $|G_1| + |G_2| - 1$ that arises from disjoint copies of $G_1$ and $G_2$ by identifying the vertex $v_1$ from $G_1$ with the vertex $v_2$ from $G_2$. Then we have
    \[
        W(G) = W(G_1) + W(G_2) + (|G_1| - 1) \, \Tr_{G_2}(v_2) + (|G_2| - 1) \, \Tr_{G_1}(v_1) .
    \]
\end{lemma}

\begin{lemma}[\hspace{1sp}{\cite{EJS1976}}]\label{wiener_path_lemma}
    For any connected graph $G$ on $n \in \mathbb{N}$ vertices, we have
    \[
        W(G) \le \binom{n + 1}{3},
    \]
    with equality holding if and only if $G \cong P_n$.
\end{lemma}

Below we prove a result about the lower bound on the degree of the vertex with the smallest transmission in a TI tree.

\begin{lemma}\label{min_deg_3_lemma}
    Let $T$ be a TI tree  with $v\in V(T)$ having the minimum transmission. Then $d_T(v) \ge 3$.
\end{lemma}
\begin{proof} Clearly, $v$ cannot be a leaf. Now, suppose that $d_T(v) = 2$  with $N_T(v)=\{x,y\}$. Let $T_x$ and $T_y$ be the subtrees of $T-v$ containing $x$ and $y$, respectively, with $n(T_x)=n_x$ and $n(T_y)=n_y$. From Lemma \ref{neighbors_lemma}, $n_x+1>n_y$ and $n_y+1>n_x$ hold simultaneously, which yields that $n_x=n_y$. Therefore, $\Tr(x)=\Tr(y)$ follows from Lemma \ref{neighbors_lemma}, which contradicts the TI property of $T$. This completes the proof of the result.
\end{proof}

In \cite{AXU}, the following result is implicitly proved.
\begin{lemma}\label{compo} Let $T$ be a TI tree with $v\in V(T)$. If $T-v=\bigcup\limits_{k=1}^{t}T_k$ with $|T_k|=a_k$ for $k\in [t]$, then $a_i\neq a_j$ for any distinct $i,j\in [t]$.\end{lemma}

\begin{lemma}\label{arm_lemma}
    Let $T_1$ be a non-path tree of order $\ell+1 >1$ with $v_1\in V(T_1)$, and let $T_2$ be a nontrivial tree with $v_2\in V(T_2)$. Suppose that $T$ is the tree obtained by identifying $v_1$ of $T_1$ and $v_2$ of $T_2$ with the new vertex labeled as $v$, and $T'$ is the tree obtained by identifying $v_2$ of $T_2$ with an endpoint of $P_{\ell+1}$, where the new vertex is labeled as $v$. Then $W(T) < W(T')$.
\end{lemma}
\begin{proof}
    Let $|T| = |T'| = n$. By Lemma~\ref{wiener_fusion_lemma}, we have
    \begin{equation}\label{aux_1}
        W(T) = W(T_2) + W(T_1) + \ell \, \Tr_{T_2}(v) + (n - \ell - 1) \Tr_{T_1}(v)
    \end{equation}
    and
    \begin{equation}\label{aux_2}
        W(T') = W(T_2) + W(P_{\ell + 1}) + \ell \, \Tr_{T_2}(v) + (n - \ell - 1) \Tr_{P_{\ell + 1}}(v) .
    \end{equation}
    By combining \eqref{aux_1} and \eqref{aux_2}, we obtain
    \[
        W(T') - W(T) = \left[ W(P_{\ell + 1}) - W(T_1) \right] + (n - \ell - 1)\left[ \Tr_{P_{\ell + 1}}(v) - \Tr_{T_1}(v) \right] .
    \]
    Since $T_1$ is not a path  of order $\ell+1$, Lemma~\ref{wiener_path_lemma} implies $W(P_{\ell + 1}) - W(T_1) > 0$. Note that $\Tr_{P_{\ell + 1}}(v) - \Tr_{T_1}(v) \ge 0$ because $\Tr_{P_{\ell + 1}}(v)$ is the greatest transmission attainable by a vertex in a connected graph of order $\ell + 1$. Therefore, $W(T') - W(T) > 0$.
\end{proof}
Similarly, based on Lemma \ref{wiener_fusion_lemma}, we obtain the following result.
\begin{corollary}\label{majorization_lemma}
    Let $T$ be a nontrivial tree with $v\in V(T)$ such that $d_T(v) \ge 3$. Assume that there are two  distinct pendent paths $P$ and $P^*$ attached to $v$, of lengths $a_1$ and $a_2$, respectively, so that $a_1 \ge a_2$. Let $T'$ be the tree obtained from $T$ by removing the leaf from $P^*$ and attaching it to the leaf of $P$. Then $W(T) < W(T')$.
\end{corollary}
\begin{remark}
    The statement of Corollary \ref{majorization_lemma} is correct even if $a_2 = 1$. In this case, there will be one less pendent path attached to $v$ in tree $T'$ than in tree $T$.
\end{remark}
Below we provide a useful formula for computing the Wiener index of trees with a few branching vertices.
\begin{lemma}[\hspace{1sp}{\cite{DEG2001,DG1977}}]\label{branching lemma}
    Let $T$ be a nontrivial tree with branching vertices $w_1,w_2,\ldots, w_k$ such that $T-w_j=\bigcup\limits_{i=1}^{d_T(w_j)}T_{ij}$ where $|T_{ij}|=\ell_{ij}$ for $i\in [d_T(w_j)]$  and $j\in [k]$.  Then
    \[
        W(T) =\binom{n + 1}{3}-\sum\limits_{j=1}^{k}\sum\limits_{1\leq p<q<r\leq d_T(w_j)}\ell_{pj}\ell_{qj}\ell_{rj}.
    \]
\end{lemma}

\section{Maximal trees of odd order}\label{sc_odd}

\ \indent In the present section we determine the extremal TI trees of any odd order attaining the maximum Wiener index. Our main result is provided in the following.

\begin{theorem}\label{odd_th}
    For every odd $n \ge 7$, there is a unique TI tree $T^*$ of order $n$ that attains the maximum Wiener index, with
    \begin{enumerate}[label=\textbf{(\roman*)}]
        \item $T^* \cong S\left( \frac{n - 1}{2}, \frac{n - 3}{2}, 1 \right)$, if $\{n - 2,n - 1\}\cap \mathcal{PS}=\emptyset$;
        \item $T^* \cong S\left( \frac{n - 1}{2}, \frac{n - 5}{2}, 2\right)$, if $\{n - 2,n - 1\}\cap \mathcal{PS}\neq\emptyset$ and $\{2n - 6, 2n - 2\}\cap\mathcal{PS}=\emptyset$;
        \item $T^* \cong C_{n - 2}\left( \frac{n - 1}{2}, \frac{n - 3}{2} + \sqrt{n - 2} \right)$, if $n - 2\in \mathcal{PS}$ and $\{2n - 6, 2n - 2\}\cap\mathcal{PS}\neq\emptyset$;
        \item $T^* \cong C_{n - 2}\left( \frac{n + 1}{2}, \frac{n - 3}{2} + \sqrt{n - 1} \right)$, if $n - 1,2n - 6\in\mathcal{PS}$.
    \end{enumerate}
\end{theorem}
\begin{remark}
    Observe that $n - 2$ and $n - 1$ cannot be a perfect square at the same time. The numbers $n - 1$ and $2n - 2$ also cannot be a perfect square at the same time. Therefore, Theorem \ref{odd_th} solves the TI tree Wiener index maximization problem for any odd order $n \ge 7$.
\end{remark}
\begin{remark}
    It is trivial to verify, e.g., by using the program \texttt{gentreeg} from the package \texttt{nauty} \cite{McKayPip2014}, that there is no TI tree of odd order $n < 7$.
\end{remark}

We begin by observing the following reformulation of an earlier result.

\begin{lemma}[\hspace{1sp}{\cite{AlKl2018}}]\label{odd_1_ti_lemma}
    For any odd $n \ge 5$, $S(\frac{n-1}{2}, \frac{n-3}{2}, 1)$ is TI if and only if $\{n - 2,n - 1\}\cap \mathcal{PS}=\emptyset$.
\end{lemma}

With this in mind, we settle case (i) of Theorem \ref{odd_th} through the following proposition.

\begin{proposition}\label{odd_th_case_1}
    For any odd $n \ge 7$ with $\{n - 2,n - 1\}\cap \mathcal{PS}=\emptyset$, $S\left( \frac{n - 1}{2}, \frac{n - 3}{2}, 1 \right)$ is the unique TI tree of order $n$ that attains the maximum Wiener index.
\end{proposition}
\begin{proof}
    Let $T$ be a TI tree of order $n$ with $v\in V(T)$ having the minimum transmission. By Lemma \ref{min_deg_3_lemma}, we have $d_T(v) \ge 3$. Let $k = d_T(v)$ and let $T_1, T_2, \ldots, T_k$ be the components of $T - v$. Also, for each $i \in [k]$, let $a_i = |T_i|$. Since $T$ is TI,  by Lemma \ref{compo}, we necessarily have $a_i \neq a_j$ for any distinct $i, j \in [k]$. Furthermore, we also have $a_i < \frac{n}{2}$ for any $i \in [k]$.  Otherwise, we have $a_i\geq \frac{n}{2}$ for some $i\in [t]$. Then there is a vertex $u\in V(T_i)$ with the minimum transmission, which contradicts the choice of $v$ in $T$. Thus, we may assume without loss of generality that $\frac{n}{2} > a_1 > a_2 > \cdots > a_k \ge 1$. Since $n$ is odd, we observe that $a_1 \le \frac{n - 1}{2}$ and $a_2 \le \frac{n - 3}{2}$.

    By repeatedly applying Lemma \ref{arm_lemma}, we have $W(T) \le W(S(a_1, a_2, \ldots, a_k))$ with equality holding if and only if $T \cong S(a_1, a_2, \ldots, a_k)$. Now, by repeatedly applying Corollary \ref{majorization_lemma},  we have
    \[
        W(S(a_1, a_2, \ldots, a_k)) \le W\left(S\left(\tfrac{n - 1}{2}, \tfrac{n - 3}{2}, 1\right)\right),
    \]
    with equality holding if and only if $k = 3$ and $(a_1, a_2, a_3) = \left(\frac{n - 1}{2}, \frac{n - 3}{2}, 1\right)$. This means that for any TI tree $T$ of order $n$ such that $T \not\cong S\left( \frac{n - 1}{2}, \frac{n - 3}{2}, 1 \right)$, we have $W(T) < W\left(S\left( \frac{n - 1}{2}, \frac{n - 3}{2}, 1 \right)\right)$. The result now follows from Lemma \ref{odd_1_ti_lemma}, which implies that $S\left( \frac{n - 1}{2}, \frac{n - 3}{2}, 1 \right)$ is TI.
\end{proof}
\begin{remark}
    Note that while applying the constructions from Lemma \ref{arm_lemma} and Corollary \ref{majorization_lemma} to obtain the maximum TI tree $S\left( \frac{n - 1}{2}, \frac{n - 3}{2}, 1 \right)$, the intermediate trees need not be TI. It only matters that the last one $S\left( \frac{n - 1}{2}, \frac{n - 3}{2}, 1 \right)$ is a  TI tree of the same order $n$.
\end{remark}

The \texttt{SageMath} \cite{SageMath} script given in Appendix \ref{sage_script} verifies that $C_9(5, 7)$ is indeed the unique TI tree of order $11$ that attains the maximum Wiener index, in accordance with case (iii) of Theorem~\ref{odd_th}. Thus, to complete the proof, we only need to deal with the TI trees of odd order $n \ge 17$ such that $n - 2$ or $n - 1$ is a perfect square. We proceed by investigating the TI property of the trees appearing in Theorem \ref{odd_th} as potential solutions.

\begin{lemma}\label{odd_2_ti_lemma}
    For any odd $n \ge 7$, the tree $S\left( \frac{n - 1}{2}, \frac{n - 5}{2}, 2 \right)$ is TI if and only if  $\mathcal{PS}\cap\{n - 4, n, \linebreak 2n-6, 2n- 2\} =\emptyset$.
\end{lemma}
\begin{proof}
    Let $z_0$ be the branching vertex of $S\left( \frac{n - 1}{2}, \frac{n - 5}{2}, 2 \right)$ and let its pendent paths be
    \[
        z_0 x_1 x_2 \cdots x_\frac{n - 1}{2}, \qquad z_0 y_1 y_2 \cdots y_\frac{n - 5}{2} \qquad \mbox{and} \qquad z_0 z_1 z_2.
    \]
    Moreover, let $\Tr(z_0) = \beta$. By repeatedly using Lemma \ref{neighbors_lemma}, we have $\Tr_G(z_1) = \beta + (n - 4)$ and $\Tr(z_2) = \beta + (2n - 6)$, alongside $\Tr(x_i) = \beta + i^2$ for any $i \in \left[\frac{n - 1}{2} \right]$ and $\Tr(y_j) = \beta + j(j + 4)$ for any $j \in \left[ \frac{n - 5}{2} \right]$. Observe that $\Tr(x_i) = \Tr(y_j)$ cannot hold. Indeed, if this is true for some $i \in \left[\frac{n - 1}{2} \right]$ and $j \in \left[ \frac{n - 5}{2} \right]$, we would get $(j + 2)^2 = i^2 + 4$, which cannot hold. Therefore, $S\left( \frac{n - 1}{2}, \frac{n - 5}{2}, 2 \right)$ is TI if and only if the two Diophantine equations $i^2 \in \{ n - 4, 2n - 6\},\, i \in \left[\frac{n - 1}{2}\right]$, and $j(j + 4) \in \{ n - 4, 2n - 6\},\, j \in \left[\frac{n - 5}{2}\right]$, have no solution. The former equation has a solution if and only if $n - 4$ or $2n - 6$ is a perfect square, while the latter has a solution if and only if $n$ or $2n - 2$ is a perfect square.
\end{proof}

\begin{lemma}\label{odd_3_ti_lemma}
    For any odd $n \ge 11$ with $n - 2\in\mathcal{PS}$, $C_{n - 2}\left(\frac{n - 1}{2}, \tfrac{n - 3}{2} + \sqrt{n - 2}\right)$ is TI.
\end{lemma}
\begin{proof}
    Let $k^2 = n - 2$ and let $z_0$ and $y_{k - 1}$ be the branching vertices of $C_{n - 2}\left(\frac{n - 1}{2}, \tfrac{n - 3}{2} + \sqrt{n - 2}\right)$ so that:
    \begin{enumerate}[label=\textbf{(\arabic*)}]
        \item $z_0 y_1 y_2 \cdots y_{k - 1}$ is the path between $z_0$ and $y_{k - 1}$;
        \item $z_0 z_1$ is the pendent path of length one from $z_0$;
        \item $z_0 x_1 x_2 \cdots x_\frac{n - 3}{2}$ is the pendent path of length $\frac{n - 3}{2}$ from $z_0$;
        \item $y_{k - 1} w$ is the pendent path of length one from $y_{k - 1}$;
        \item $y_{k - 1} y_k y_{k + 1} \cdots y_\frac{n - 3}{2}$ is the pendent path of length $\frac{n - 1}{2} - k$ from $y_{k - 1}$;
    \end{enumerate}
    see Figure \ref{fig_1}. Moreover, let $\Tr(z_0) = \beta$. By repeatedly using Lemma~\ref{neighbors_lemma}, we may conclude that $\Tr(z_1) = \beta + k^2$, $\Tr(x_i) = \beta + i(i + 2)$ for any $i \in \left[ \frac{n - 3}{2} \right]$ and $\Tr(y_j) = \beta + j^2$ for any $j \in [k - 1]$, alongside $\Tr(w) = \beta + (2k^2 - 2k + 1)$ and $\Tr(y_j) = \beta + j(j + 2) - 2k + 2$ for any $j \in \left\{ k, k + 1, \ldots, \frac{n - 3}{2} \right\}$. Since $\Tr(w) > \Tr(z_1)$, to prove that $C_{n - 2}\left(\frac{n - 1}{2}, \tfrac{n - 3}{2} + \sqrt{n - 2}\right)$ is TI, it suffices to show that:
    \begin{enumerate}[label=\textbf{(\alph*)}]
        \item $\Tr(x_i)\neq \Tr(y_j)$ for any $i \in \left[ \frac{n - 3}{2} \right]$ and $j \in \left[ \frac{n - 3}{2} \right]$;
        \item none of the $\Tr(x_i)$ values are equal to $\Tr(z_1)$ or $\Tr(w)$, for any $i \in \left[ \frac{n - 3}{2} \right]$;
        \item none of the $\Tr(y_j)$ values are equal to $\Tr(z_1)$ or $\Tr(w)$, for any $j \in \left[ \frac{n - 3}{2} \right]$.
    \end{enumerate}

\begin{figure}[t]
\centering
\begin{tikzpicture}[scale=1.0]
\tikzstyle{vertex}=[draw,circle,minimum size=4pt,inner sep=1pt,fill=black]
\tikzstyle{edge}=[draw,thick]
\tikzstyle{dedge}=[draw,thick,dashed]

\node[vertex, label=below:$z_0$] (0) at (0, 0) {};
\node[vertex, label=right:$z_1$] (1) at (0, 1) {};

\node[vertex, label=below:$y_1$] (2) at (1, 0) {};
\node[vertex, label=below:$y_2$] (3) at (2, 0) {};
\node[vertex, label=below:$y_{k-2}$] (4) at (4, 0) {};
\node[vertex, label=below:$y_{k-1}$] (5) at (5, 0) {};
\node[vertex, label=right:$w$] (6) at (5, 1) {};
\node[vertex, label=below:$y_{k}$] (7) at (6, 0) {};
\node[vertex, label=below:$y_\frac{n - 3}{2}$] (8) at (8, 0) {};

\node[vertex, label=below:$x_1$] (9) at (-1, 0) {};
\node[vertex, label=below:$x_2$] (10) at (-2, 0) {};
\node[vertex, label=below:$x_\frac{n - 3}{2}$] (11) at (-4, 0) {};

\path[edge] (0) -- (1);
\path[edge] (0) -- (2);
\path[edge] (2) -- (3);
\path[dedge] (3) -- (4);
\path[edge] (4) -- (5);
\path[edge] (5) -- (6);
\path[edge] (5) -- (7);
\path[dedge] (7) -- (8);
\path[edge] (0) -- (9);
\path[edge] (9) -- (10);
\path[dedge] (10) -- (11);
\end{tikzpicture}
\caption{The graph $C_{n - 2}\left(\frac{n - 1}{2}, \tfrac{n - 3}{2} + \sqrt{n - 2}\right)$ from Lemma \ref{odd_3_ti_lemma}.}
\label{fig_1}
\end{figure}
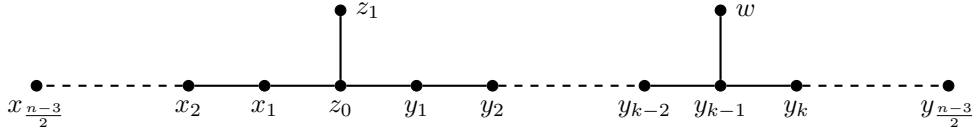

    Suppose that $\Tr(x_i) = \Tr(y_j)$ holds for some $i \in \left[ \frac{n - 3}{2} \right]$ and $j \in [k - 1]$. In this case, we get $i(i + 2) = j^2$. However, this is equivalent to $(i + 1)^2 = j^2 + 1$, which cannot be satisfied. Now, suppose that $\Tr(x_i) = \Tr(y_j)$ holds for some $i \in \left[ \frac{n - 3}{2} \right]$ and $j \in \{ k, k + 1, \ldots, \frac{n - 3}{2} \}$. In this case, we obtain $(j - i)(j + i + 2) = 2k - 2$. Clearly, $j - i$ must be even, hence $j - i \ge 2$. Since $j \ge k$ and $i \ge 1$, we obtain $(j - i)(j + i + 2) \ge 2(k + 3) = 2k + 6$, yielding a contradiction. These two contradictions imply that condition~(a) is satisfied.

    If  $\Tr(x_i) = \Tr(z_1)$ for some $i \in \left[ \frac{n - 3}{2} \right]$, we get a contradiction because $i(i + 2) = k^2$ is equivalent to $(i + 1)^2 = k^2 + 1$, which clearly cannot be true. On the other hand, if we have $\Tr(x_i) = \Tr(w)$ for some $i \in \left[ \frac{n - 3}{2} \right]$, we then obtain $i(i + 2) = 2k^2 - 2k + 1$, which is equivalent to $2(i + 1)^2 = (2k - 1)^2 + 3$. Since $i$ is odd, we have $8 \mid 2(i + 1)^2$, hence $(2k - 1)^2 \equiv 5 \pmod{8}$, which is impossible because $5$ is not a quadratic residue modulo $8$. Therefore, condition (b) holds.

    We trivially observe that $\Tr(y_j) = \Tr(z_1)$ cannot be true for any $j \in \left[ \frac{n - 3}{2} \right]$, since $\Tr(y_{k - 1}) < \Tr(z_1) < \Tr(y_k)$. Furthermore, $\Tr(y_j) = \Tr(w)$ does not hold for any $j \in [k - 1]$ because $\Tr(w) > \Tr(y_{k - 1})$. Suppose that $\Tr(y_j) = \Tr(w)$ for some $j \in \left\{ k, k + 1, \ldots, \frac{n - 3}{2} \right\}$. In this case, we get $j(j + 2) - 2k + 2 = 2k^2 - 2k + 1$, which is equivalent to $(j + 1)^2 = 2k^2$. This equality cannot hold, hence we conclude that condition (c) is satisfied.
\end{proof}

\begin{lemma}\label{odd_4_ti_lemma}
    For any odd $n \ge 17$ with $n - 1\in\mathcal{PS}$, $C_{n - 2}\left(\frac{n + 1}{2}, \tfrac{n - 3}{2} + \sqrt{n - 1}\right)$ is TI.
\end{lemma}
\begin{proof}
    Let $k^2 = n - 1$ and let $z_0$ and $y_{k - 2}$ be the branching vertices of $C_{n - 2}\left(\frac{n + 1}{2}, \tfrac{n - 3}{2} + \sqrt{n - 1}\right)$ so that:
    \begin{enumerate}[label=\textbf{(\arabic*)}]
        \item $z_0 y_1 y_2 \cdots y_{k - 2}$ is the path between $z_0$ and $y_{k - 2}$;
        \item $z_0 z_1$ is the pendent path of length one from $z_0$;
        \item $z_0 x_1 x_2 \cdots x_\frac{n - 1}{2}$ is the pendent path of length $\frac{n - 1}{2}$ from $z_0$;
        \item $y_{k - 2} w$ is the pendent path of length one from $y_{k - 2}$;
        \item $y_{k - 2} y_{k - 1} y_k \cdots y_\frac{n - 5}{2}$ is the pendent path of length $\frac{n - 1}{2} - k$ from $y_{k - 2}$;
    \end{enumerate}
    see Figure \ref{fig_2}. Moreover, let $\Tr(z_0) = \beta$. By repeatedly using Lemma~\ref{neighbors_lemma}, we obtain $\Tr(z_1) = \beta + (k^2 - 1)$, $\Tr(x_i) = \beta + i^2$ for any $i \in \left[ \frac{n - 1}{2} \right]$ and $\Tr(y_j) = \beta + j(j + 2)$ for any $j \in [k - 2]$, alongside $\Tr(w) = \beta + (2k^2 - 2k - 1)$ and $\Tr(y_j) = \beta + j(j + 4) - 2k + 4$ for any $j \in \left\{ k - 1, k, \ldots, \frac{n - 5}{2} \right\}$. Since $\Tr(w) > \Tr(z_1)$, to prove that $C_{n - 2}\left(\frac{n + 1}{2}, \tfrac{n - 3}{2} + \sqrt{n - 1}\right)$ is TI, it suffices to show that:
    \begin{enumerate}[label=\textbf{(\alph*)}]
        \item $\Tr(x_i) = \Tr(y_j)$ cannot hold for any $i \in \left[ \frac{n - 1}{2} \right]$ and $j \in \left[ \frac{n - 5}{2} \right]$;
        \item none of the $\Tr(x_i)$ values are equal to $\Tr(z_1)$ or $\Tr(w)$, for any $i \in \left[ \frac{n - 1}{2} \right]$;
        \item none of the $\Tr(y_j)$ values are equal to $\Tr(z_1)$ or $\Tr(w)$, for any $j \in \left[ \frac{n - 5}{2} \right]$.
    \end{enumerate}

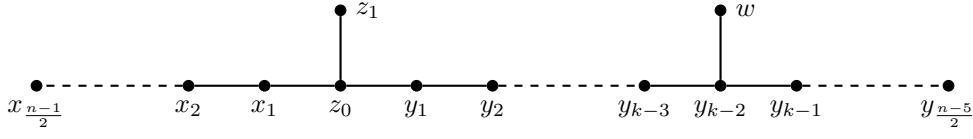
\begin{figure}[t]
\centering
\begin{tikzpicture}[scale=1.0]
\tikzstyle{vertex}=[draw,circle,minimum size=4pt,inner sep=1pt,fill=black]
\tikzstyle{edge}=[draw,thick]
\tikzstyle{dedge}=[draw,thick,dashed]

\node[vertex, label=below:$z_0$] (0) at (0, 0) {};
\node[vertex, label=right:$z_1$] (1) at (0, 1) {};

\node[vertex, label=below:$y_1$] (2) at (1, 0) {};
\node[vertex, label=below:$y_2$] (3) at (2, 0) {};
\node[vertex, label=below:$y_{k-3}$] (4) at (4, 0) {};
\node[vertex, label=below:$y_{k-2}$] (5) at (5, 0) {};
\node[vertex, label=right:$w$] (6) at (5, 1) {};
\node[vertex, label=below:$y_{k-1}$] (7) at (6, 0) {};
\node[vertex, label=below:$y_\frac{n - 5}{2}$] (8) at (8, 0) {};

\node[vertex, label=below:$x_1$] (9) at (-1, 0) {};
\node[vertex, label=below:$x_2$] (10) at (-2, 0) {};
\node[vertex, label=below:$x_\frac{n - 1}{2}$] (11) at (-4, 0) {};

\path[edge] (0) -- (1);
\path[edge] (0) -- (2);
\path[edge] (2) -- (3);
\path[dedge] (3) -- (4);
\path[edge] (4) -- (5);
\path[edge] (5) -- (6);
\path[edge] (5) -- (7);
\path[dedge] (7) -- (8);
\path[edge] (0) -- (9);
\path[edge] (9) -- (10);
\path[dedge] (10) -- (11);
\end{tikzpicture}
\caption{The graph $C_{n - 2}\left(\frac{n + 1}{2}, \tfrac{n - 3}{2} + \sqrt{n - 1}\right)$ from Lemma \ref{odd_4_ti_lemma}.}
\label{fig_2}
\end{figure}

    If we suppose that $\Tr(x_i) = \Tr(y_j)$ holds for some $i \in \left[ \frac{n - 1}{2} \right]$ and $j \in [k - 2]$, we then obtain $i^2 = j(j + 2)$, which is equivalent to $i^2 + 1 = (j + 1)^2$ and thus cannot be true. Now, suppose that $\Tr(x_i) = \Tr(y_j)$ holds for some $i \in \left[ \frac{n - 1}{2} \right]$ and $j \in \left\{ k - 1, k, \ldots, \frac{n - 5}{2} \right\}$. In this case, we get $(j - i + 2)(j + i + 2) = 2k$. Since $j - i + 2$ must be even, we have $j - i + 2 \ge 2$. From $j \ge k - 1$ and $i \ge 1$, we obtain $(j - i + 2)(j + i + 2) \ge 2(k + 2) = 2k + 4$, yielding a contradiction. Therefore, condition (a) is satisfied.

    We clearly have $\Tr(x_i) \neq \Tr(z_1)$ for every $i \in \left[ \frac{n - 1}{2} \right]$, since otherwise, we would get $i^2 = k^2 - 1$, which is impossible. Suppose that $\Tr(x_i) = \Tr(w)$ for some $i \in \left[ \frac{n - 1}{2} \right]$. In this case, we have $i^2 = 2k^2 - 2k - 1$, which is equivalent to $2i^2 = (2k - 1)^2 - 3$. Since $i$ is odd, we have $2i^2 \equiv 2 \pmod{8}$, hence $(2k - 1)^2 \equiv 5 \pmod{8}$, which is impossible because $5$ is not a quadratic residue modulo $8$. Thus, condition (b) holds.

    We trivially have $\Tr(y_j) \neq \Tr(z_1)$ for every $j \in \left[ \frac{n - 5}{2} \right]$ because $\Tr(y_{k - 2}) < \Tr(z_1) < \Tr(y_{k - 1})$. Also, $\Tr(y_j) = \Tr(w)$ cannot hold for any $j \in [k - 2]$, since $\Tr(w) > \Tr(y_{k - 2})$. Now, suppose that $\Tr(y_j) = \Tr(w)$ for some $j \in \left\{ k - 1, k, \ldots, \frac{n - 5}{2} \right\}$. In this case, we have $j(j + 4) - 2k + 4 = 2k^2 - 2k - 1$, which is equivalent to $(j + 2)^2 = 2k^2 - 1$. Since $n$ is odd, we know that $k$ must be even, hence $8 \mid 2k^2$. This means that $(j + 2)^2 \equiv 7 \pmod{8}$, which is not possible because $7$ is not a quadratic residue modulo $8$. From here, it follows that condition~(c) is satisfied.
\end{proof}

Combining Lemma \ref{branching lemma} with some elementary calculations, we can get the following results, omitting the respective  proofs.
\begin{lemma}\label{odd_2_wiener_lemma}
    For any odd $n \ge 7$, we have
    \[
        W\left(S\left( \tfrac{n - 1}{2}, \tfrac{n - 5}{2}, 2 \right)\right) = \frac{n^3 - 3n^2 + 17n - 15}{6} .
    \]
\end{lemma}
\begin{lemma}\label{odd_5_wiener_lemma}
    For any odd $n \ge 9$, we have
    \[
        W\left(S\left( \tfrac{n - 1}{2}, \tfrac{n - 7}{2}, 3 \right)\right) = \frac{2n^3 - 9n^2 + 70n - 63}{12}.
    \]
\end{lemma}
\begin{lemma}\label{odd_3_wiener_lemma}
    For any odd $n \ge 11$ with $n - 2\in\mathcal{PS}$, we have
    \[
        W\left(C_{n - 2}\left(\tfrac{n - 1}{2}, \tfrac{n - 3}{2} + \sqrt{n - 2} \right)\right) = \frac{n^3 - 3n^2 + 17n  - 6 \sqrt{n - 2} - 21}{6} .
    \]
\end{lemma}
\begin{lemma}\label{odd_4_wiener_lemma}
    For any odd $n \ge 17$ with $n - 1\in\mathcal{PS}$, we have
    \[
        W\left(C_{n - 2}\left( \tfrac{n + 1}{2}, \tfrac{n - 3}{2} + \sqrt{n - 1} \right)\right) = \frac{n^3 - 3n^2 + 17n  - 6 \sqrt{n - 1} - 15}{6}.
    \]
\end{lemma}

We proceed with the following auxiliary lemma.

\begin{lemma}\label{odd_char_lemma}
    Let $n \ge 17$ be an odd integer with $\{n - 2,n - 1\}\cap\mathcal{PS}\neq\emptyset$. Also, let $T$ be a TI tree of order $n$ that attains the maximum Wiener index with $v\in V(T)$ having the minimum transmission. Then $d_T(v) = 3$ and one of the following holds:
    \begin{enumerate}[label=\textbf{(\alph*)}]
        \item there is a leaf attached to $v$, while the other two components of $T - v$ have $\frac{n - 1}{2}$ and $\frac{n - 3}{2}$ vertices, respectively;
        \item there is a pendent path of length two attached to $v$, while the other two components of $T - v$ have $\frac{n - 1}{2}$ and $\frac{n - 5}{2}$ vertices, respectively.
    \end{enumerate}
\end{lemma}
\begin{proof}
    By Lemma \ref{min_deg_3_lemma}, we have $d_T(v) \ge 3$. Let $k = d_T(v)$ and let $T_1, T_2, \ldots, T_k$ be the components of $T - v$. Also, for each $i \in [k]$, let $a_i = |T_i|$. Analogously to Proposition~\ref{odd_th_case_1}, we may assume without loss of generality that $\frac{n - 1}{2} \ge a_1 > a_2 > \cdots > a_k \ge 1$.

    Now, suppose that $\sum_{i = 3}^k a_i \ge 3$. In this case, by repeatedly applying Lemma~\ref{arm_lemma}, then Corollary \ref{majorization_lemma}, as in Proposition \ref{odd_th_case_1}, we may conclude that $W(T) \le W\left(S\left(\frac{n - 1}{2}, \frac{n - 7}{2}, 3\right)\right)$. From Lemma \ref{odd_5_wiener_lemma}, we then obtain
    \begin{equation}\label{aux_5}
        W(T) \le \frac{2n^3 - 9n^2 + 70n - 63}{12}.
    \end{equation}
    If $n - 2\in\mathcal{PS}$, then Lemma \ref{odd_3_ti_lemma} implies that $C_{n - 2}\left(\frac{n - 1}{2}, \tfrac{n - 3}{2} + \sqrt{n - 2}\right)$ is a TI tree of order $n$, while if $n - 1\in\mathcal{PS}$, then $C_{n - 2}\left(\frac{n + 1}{2}, \tfrac{n - 3}{2} + \sqrt{n - 1}\right)$ is a TI tree of order $n$ from Lemma \ref{odd_4_ti_lemma}. Either way, since $T$ attains the maximum Wiener index, while $6 \sqrt{n - 2} + 21 > 6 \sqrt{n - 1} + 15$, by Lemmas~\ref{odd_3_wiener_lemma} and \ref{odd_4_wiener_lemma}, we have
    \begin{equation}\label{aux_6}
        W(T) \ge \frac{n^3 - 3n^2 + 17n  - 6 \sqrt{n - 2} - 21}{6} .
    \end{equation}
    By combining \eqref{aux_5} and \eqref{aux_6}, we get
    \[
        3n^2 - 36n - 12 \sqrt{n - 2} + 21 \le 0, \qquad \mbox{i.e.,} \qquad n^2 - 12n - 4 \sqrt{n - 2} + 7 \le 0,
    \]
    which obviously cannot hold when $n \ge 17$.

    From the obtained contradiction, we conclude that $\sum_{i = 3}^k a_i \le 2$, which means that $k = 3$ and either $a_3 = 1$ or $a_3 = 2$. Also, note that $a_1 \le \frac{n - 1}{2}$ and $a_2 \le \frac{n - 3}{2}$. Therefore, if $a_3 = 1$, then there is a leaf attached to $v$, while the other two components of $T - v$ have $\frac{n - 1}{2}$ and $\frac{n - 3}{2}$ vertices. On the other hand, if $a_3 = 2$, then there is a pendent path of length two attached to $v$, while the other two components of $T - v$ have $\frac{n - 1}{2}$ and $\frac{n - 5}{2}$ vertices.
\end{proof}

We will say that a TI tree $T$ is of \textit{type A} (resp.\ \textit{type B}) if its vertex with the minimum transmission $v$ is of degree three and condition (a) (resp.\ (b)) holds from Lemma \ref{odd_char_lemma}. With this in mind, Lemma \ref{odd_char_lemma} states that any extremal tree is of type A or type B. By investigating these two types of trees separately, we obtain the next result.

\begin{lemma}\label{odd_type_a_1}
    For any odd $n \ge 11$ with $n - 2\in \mathcal{PS}$, $C_{n - 2}\left(\frac{n - 1}{2}, \frac{n - 3}{2} + \sqrt{n - 2} \right)$ is the unique TI tree of type A of order $n$ that attains the greatest Wiener index.
\end{lemma}
\begin{proof}
    Let $T$ be a TI tree of type A of order $n$ with $v\in V(T)$ having the minimum transmission. Also,  $T-v=\bigcup\limits_{i=1}^{3}T_i$ with $|T_1| = \frac{n - 1}{2}$, $|T_2| = \frac{n - 3}{2}$ and $|T_3| = 1$. Moreover, let $u$ be the leaf attached to $v$.  If $T_1$  is a path from $v$, then the vertex on this path with distance $\sqrt{n - 2}$ to $v$ has the same transmission in $T$ as $u$. A similar conclusion can be made if the branching vertex in $T$ from $T_1$ closest to $v$ has a distance to $v$ of at least $\sqrt{n - 2}$. Thus, if we denote by $w$ the branching vertex in $T$ from $T_1$ closest to $v$, we have $d_T(w, v) \le \sqrt{n - 2} - 1$.

    By repeatedly using Lemma \ref{arm_lemma} to $T_2$ and the components of $T - w$ not containing $v$, we get $W(T) \le W(T')$, where $T'$ is a double starlike tree with branching vertices $v$ and $w$ such that:
    \begin{enumerate}[label=\textbf{(\arabic*)}]
        \item $d_{T'}(w, v) = d_T(w, v) \le \sqrt{n - 2} - 1$;
        \item $d_{T'}(v) = 3$ and the two pendent paths attached to $v$ have lengths $\frac{n - 3}{2}$ and $1$, respectively.
    \end{enumerate}
    Of course, equality holds if and only if $T \cong T'$. Now, by repeatedly applying Corollary \ref{majorization_lemma} to the pendent paths of $w$ in $T'$, we conclude that $W(T') \le W(T'')$, where $T'' \cong C_{n - 2}\left(\frac{n - 1}{2}, \frac{n - 1}{2} + t\right)$ and $t = d_{T'}(w, v)$, with equality holding if and only if $T' \cong T''$.

    From Lemma \ref{wiener_fusion_lemma}, we trivially observe that for any $\frac{n + 1}{2} \le t_1 < t_2 \le n - 3$, we have
    \[
        W\left(C_{n - 2}\left(\tfrac{n - 1}{2}, t_1 \right)\right) < W\left(C_{n - 2}\left(\tfrac{n - 1}{2}, t_2 \right) \right).
    \]
    With this in mind, it follows that $W(T'') \le W\left( C_{n - 2}\left( \frac{n - 1}{2}, \frac{n - 3}{2} + \sqrt{n - 2} \right) \right)$. Moreover, we get $W(T) < W\left( C_{n - 2}\left( \frac{n - 1}{2}, \frac{n - 3}{2} + \sqrt{n - 2} \right) \right)$ for any TI tree $T$ of type A of order $n$ which is not isomorphic to $C_{n - 2} \left( \frac{n - 1}{2}, \frac{n - 3}{2} + \sqrt{n - 2}\right)$. The result now follows from Lemma \ref{odd_3_ti_lemma}.
\end{proof}

The next result can be proved in an analogous manner, hence we omit its proof.

\begin{lemma}\label{odd_type_a_2}
    For any odd $n \ge 17$ with $n - 1\in\mathcal{PS}$, $C_{n - 2}\left(\frac{n + 1}{2}, \frac{n - 3}{2} + \sqrt{n - 1} \right)$ is the unique TI tree of type A of order $n$ that attains the greatest Wiener index.
\end{lemma}

We are now in a position to settle case (ii) of Theorem \ref{odd_th} as follows.

\begin{proposition}\label{odd_th_case_2}
    For any odd $n \ge 17$ with $\{n - 2,n - 1\}\cap\mathcal{PS}\neq\emptyset$ and $\{2n - 6,2n - 2\}\cap\mathcal{PS}=\emptyset$, $S\left( \frac{n - 1}{2}, \frac{n - 5}{2}, 2\right)$ is the unique TI tree of order $n$ that attains the maximum Wiener index.
\end{proposition}
\begin{proof}
    Since $n - 2$ or $n - 1$ is a perfect square, it trivially follows that neither $n - 4$ nor $n$ is a perfect square. Hence, together with the assumption here, Lemma \ref{odd_2_ti_lemma} implies that $S\left( \frac{n - 1}{2}, \frac{n - 5}{2}, 2\right)$ is TI. Moreover, from Lemmas~\ref{odd_2_wiener_lemma}, \ref{odd_3_wiener_lemma}, \ref{odd_4_wiener_lemma}, \ref{odd_type_a_1} and \ref{odd_type_a_2}, it follows that $S\left( \frac{n - 1}{2}, \frac{n - 5}{2}, 2\right)$ has a greater Wiener index than any TI tree of type A of order $n$. Therefore, to complete the proof, it suffices to show that $S\left( \frac{n - 1}{2}, \frac{n - 5}{2}, 2\right)$ is the unique TI tree of order $n$ with the greatest Wiener index among the TI trees of type B. Let $T$ be any TI tree of type B of order $n$. By applying Lemma \ref{arm_lemma} and Corollary \ref{majorization_lemma} in the same way as in Proposition \ref{odd_th_case_1}, we have $W(T) \le W\left( S\left( \frac{n - 1}{2}, \frac{n - 5}{2}, 2\right) \right)$, with equality holding if and only if $T \cong S\left( \frac{n - 1}{2}, \frac{n - 5}{2}, 2\right)$.
\end{proof}

Of course, if $2n - 6$ or $2n - 2$ is a perfect square, then $S\left( \frac{n - 1}{2}, \frac{n - 5}{2}, 2\right)$ cannot be the solution because it is not TI, as shown in Lemma \ref{odd_2_ti_lemma}. In this case, a different argumentation should be used to obtain the solution. We thus finalize the proof of Theorem \ref{odd_th} through the following proposition.

\begin{proposition}\label{odd_th_cases_3_4}
   Let $n \ge 17$ be odd with $\{n - 2,n - 1\}\cap\mathcal{PS}\neq\emptyset$ and $\{2n - 6,2n - 2\}\cap\mathcal{PS}\neq\emptyset$. Then $C_{n - 2}\left( \frac{n - 1}{2}, \frac{n - 3}{2} + \sqrt{n - 2} \right)$ is the unique TI tree of order $n$  attaining the greatest Wiener index if $n - 2\in\mathcal{PS}$, and $C_{n - 2}\left( \frac{n + 1}{2}, \frac{n - 3}{2} + \sqrt{n - 1} \right)$ is the unique TI tree of order $n$  attaining the greatest Wiener index if $n - 1\in\mathcal{PS}$.
\end{proposition}
\begin{proof}
    Since $6 \sqrt{n - 2} + 21 > 6 \sqrt{n - 1} + 15$, Lemmas~\ref{odd_3_wiener_lemma}, \ref{odd_4_wiener_lemma}, \ref{odd_type_a_1} and \ref{odd_type_a_2} imply that to complete the proof, it is enough to show that every TI tree $T$ of type B of order $n$ satisfies
    \begin{equation}\label{aux_7}
        W(T) < \frac{n^3 - 3n^2 + 17n  - 6 \sqrt{n - 2} - 21}{6} .
    \end{equation}
    We will now provide the proof only for the case when $2n - 6$ is a perfect square, since the case when $2n - 2$ is a perfect square can be resolved analogously.

    Let $T$ be any TI tree of type B of order $n$ with $v\in V(T)$ having the minimum transmission. Also, let $T_1$, $T_2$ and $T_3$ be the components of $T - v$ so that $|T_1| = \frac{n - 1}{2}$, $|T_2| = \frac{n - 5}{2}$ and $|T_3| = 2$. Moreover, let $v u_1 u_2$ be the pendent path of length two attached to $v$. If $T_1$ is a path from $v$, then the vertex on this path with distance $\sqrt{2n - 6}$ to $v$ has the same transmission in $T$ as $u_2$. A similar conclusion can be made if the branching vertex in $T$ from $T_1$ closest to $v$ has a distance to $v$ of at least $\sqrt{2n - 6}$. Therefore, if we denote by $w$ the branching vertex in $T$ from $T_1$ closest to $v$, we have $d_T(w, v) \le \sqrt{2n - 6} - 1$.

    By applying Lemma \ref{arm_lemma} and Corollary \ref{majorization_lemma} in the same manner as in Lemma \ref{odd_type_a_1}, we conclude that $W(T) \le W(T')$, where $T'$ arises from $S(\frac{n - 3}{2}, \frac{n - 5}{2}, 2)$ by attaching a leaf to the vertex on the pendent path of length $\frac{n - 3}{2}$ with the distance $\sqrt{2n - 6} - 1$ from the  branching vertex in it.  By Lemma \ref{branching lemma}, we have
    \[
        W(T') = \frac{2n^3 - 9n^2 + 70n - 12 \sqrt{2n - 6} - 111}{12}.
    \]
    Therefore, to obtain \eqref{aux_7}, it suffices to prove that
    \[
        W(T') < \frac{n^3 - 3n^2 + 17n  - 6 \sqrt{n - 2} - 21}{6},
    \]
    i.e.,
    \[
        n^2 - 12n - 4\sqrt{n - 2} + 4\sqrt{2n - 6} + 23 > 0,
    \]
    which trivially holds when $n \ge 17$.
\end{proof}

Theorem \ref{odd_th} now follows from Propositions \ref{odd_th_case_1}, \ref{odd_th_case_2} and \ref{odd_th_cases_3_4}, as well as the computational results obtained by the \texttt{SageMath} script given in Appendix \ref{sage_script}.

\section{Maximal trees of even order}\label{sc_even}

\ \indent In this section we solve the Wiener index maximization problem on the set of TI trees of order $n$, for any even $n$, apart from the values $n \ge 30$ with $4n - 7 \in \mathcal{PS}$, or $\{4n - 15, 4n - 7\} \cap \mathcal{PS} = \emptyset$ and $8n - 15 \in \mathcal{PS}$. We begin with the next proposition that can easily be verified by using the \texttt{SageMath} script given in Appendix \ref{sage_script}.

\begin{proposition}\label{even_secondary} For $n\in \{14,22,24\}$, the tree $T^*$ uniquely attains the maximum Wiener index among all TI trees of order $n$, where $T^*\cong \begin{cases}
C_9(3, 1; 5, 1; 5, 3), & n=14;\\
C_{17}(11, 2; 13, 3), & n=22;\\
C_{21}(11, 2; 12, 1), & n=24.
\end{cases}$
\end{proposition}

 In view of Proposition \ref{even_secondary},  our main result is embodied in the following.

\begin{theorem}\label{even_th}
   Let  $n\ge 16$ with $n\notin\{22,24\}$, and suppose that neither of the following two conditions holds:
    \begin{enumerate}[label=\textbf{(\arabic*)}]
        \item $4n - 7 \in \mathcal{PS}$;
        \item $\{4n - 15, 4n - 7\} \cap \mathcal{PS} = \emptyset$ and $8n - 15 \in \mathcal{PS}$.
    \end{enumerate}
    Then there is a unique TI tree $T^*$ of order $n$ that attains the maximum Wiener index, with
    \begin{enumerate}[label=\textbf{(\roman*)}]
        \item $T^* \cong S\left( \frac{n}{2} - 1, \frac{n}{2} - 2, 2 \right)$, if  $\{4n - 15,4n - 7,8n - 23, 8n - 15\}\cap\mathcal{PS}=\emptyset$;
        \item $T^* \cong C_{n - 3}\left(\tfrac{n}{2} - 1, 2; \frac{n}{2} + \tfrac{\sqrt{4n - 15} - 1}{2} - 2, 1\right)$, if $4n - 15\in\mathcal{PS}$ and $\{8n - 23, 8n - 11 - 4 \sqrt{4n - 15},8n - 27 + 4 \sqrt{4n - 15}\}\cap\mathcal{PS}=\emptyset$;
        \item $T^* \cong C_{n - 3}\left(\tfrac{n}{2} - 1, 2; \frac{n}{2} + \tfrac{\sqrt{4n - 15} - 1}{2} - 3, 1\right)$, if $4n - 15\in\mathcal{PS}$ and $\{8n - 23, 8n - 11 - 4 \sqrt{4n - 15},8n - 27 + 4 \sqrt{4n - 15}\}\cap\mathcal{PS}\neq\emptyset$;
        \item $T^* \cong C_{n - 3}\left( \tfrac{n}{2} - 1, 2; \tfrac{n}{2} + \tfrac{\sqrt{8n - 23} - 1}{2} - 2, 1\right)$, if $\{ 4n - 15, 4n - 7 \} \cap \mathcal{PS} = \emptyset$ and $8n - 23 \in \mathcal{PS}$.
    \end{enumerate}
\end{theorem}
\begin{remark}
   Note that $4n - 15$ and $4n - 7$ cannot be a perfect square at the same time, and neither can $8n - 23$ and $8n - 15$. Also, observe that if $n\ge 14$ is even with $4n - 7 \in \mathcal{PS}$, then we have $n \in \{14, 22\}$ or $n \ge 32$. On the other hand, if $n\geq 14$ is even with $\{4n - 15, 4n - 7\} \cap \mathcal{PS} = \emptyset$ and $8n - 15 \in \mathcal{PS}$, then we get $n \ge 30$. With all of this in mind, Proposition~\ref{even_secondary} and Theorem~\ref{even_th} solve the TI tree Wiener index maximization problem for any even order $n \ge 14$, apart from the values $n \ge 30$ such that $4n - 7 \in \mathcal{PS}$, or $\{4n - 15, 4n - 7\} \cap \mathcal{PS} = \emptyset$ and $8n - 15 \in \mathcal{PS}$.
\end{remark}
\begin{remark}
    It can be trivially checked via the program \texttt{gentreeg} from the package \texttt{nauty} that there is no TI tree of even order $n < 14$.
\end{remark}

We begin with the following short lemma.

\begin{lemma}\label{even lemma-1}
    For any even $n \ge 14$, $S\left( \frac{n}{2} - 1, \frac{n}{2} - 2, 2 \right)$ is $TI$ if and only if $\{ 4n - 15, 4n - 7, \linebreak 8n - 23, 8n - 15 \} \cap \mathcal{PS} = \emptyset$.
\end{lemma}
\begin{proof}
     Let $T=S\left( \frac{n}{2} - 1, \frac{n}{2} - 2, 2 \right)$ with $v\in V(T)$ of degree $3$. We trivially observe that $v$ has the smallest transmission in $T$. Now, assume that $\Tr(v)=\beta$. From the structure of $T$, by Lemma~\ref{neighbors_lemma}, we have $\Tr(T)-\beta = A_0\cup A_1\cup A_2$ where $A_0=\{0,n-4,2n-6\}$, $A_1 = \{ x^2 + 3x \mid x\in[\frac{n}{2} - 2]\}$ and $A_2 = \{ y^2 + y \mid y \in [\frac{n}{2} - 1]\}$.

    We first claim that $A_1\cap A_2=\emptyset$. If $x^2+3x=y^2+y$ with $x\in[\frac{n}{2} - 2]$ and $y\in [\frac{n}{2} - 1]$, we have $y>x$ with $(y-x)(y+x+1)=2x$. Thus $(y-x)(y+x+1)\geq 2x+2>2x$ as a clear contradiction. Therefore, this claim holds, which means that $T$ is TI if and only if $A_0 \cap A_1 = \emptyset$ and $A_0 \cap A_2 = \emptyset$. This holds, by some routine checking, if and only if none of $4n - 15$, $4n - 7$, $8n - 23$ and $8n - 15$ is a perfect square.
\end{proof}

With this in mind, we settle case (i) of Theorem \ref{even_th} through the following proposition.

\begin{proposition}\label{even_th_case_1}
    For any even $n \ge 16$ such that $\{4n - 15, 4n - 7, 8n - 23, 8n - 15\} \cap \mathcal{PS} = \emptyset$, $S\left( \frac{n}{2} - 1, \frac{n}{2} - 2, 2 \right)$ is the unique TI tree of order $n$ that attains the maximum Wiener index.
\end{proposition}
\begin{proof}
     Assume that $T$ is a TI tree of even order $n \geq 16$ with the maximum Wiener index with $\{ 4n - 15, 4n - 7, 8n - 23, 8n - 15 \} \cap \mathcal{PS} = \emptyset$. Let $v\in V(T)$ have the smallest transmission in $T$. By Lemmas \ref{compo}, \ref{arm_lemma} and \ref{even lemma-1}, we conclude that $T$ is a starlike tree of maximum degree $3$ with the shortest arm of length $2$, completing the proof of the result.
\end{proof}

The \texttt{SageMath} script given in Appendix \ref{sage_script} verifies that $C_{13}(7, 2; 9, 1)$ is indeed the unique TI tree of order $16$ that attains the maximum Wiener index, in accordance with case (ii) of Theorem~\ref{even_th}. Therefore, in the remainder of the section we can assume that $n \neq 16$. We proceed by investigating the TI property of the remaining trees appearing in Theorem \ref{even_th} as potential solutions.

\begin{lemma}\label{even_sol_2}
    Let $n \ge 14$ be even and suppose that $4n - 15\in \mathcal{PS}$. Also, let $k = \frac{\sqrt{4n - 15} - 1}{2}$. Then $C_{n - 3}\left(\frac{n}{2} - 1, 2; \frac{n}{2} + k - 2, 1\right)$ is TI if and only if $\{8k^2 + 17,8k^2 + 8k + 9,8k^2 + 16k + 9\}\cap \mathcal{PS}=\emptyset$.
\end{lemma}
\begin{proof}
    Let $z_0$ and $y_{k - 1}$ be the branching vertices of $C_{n - 3}\left(\frac{n}{2} - 1, 2; \frac{n}{2} + k - 2, 1\right)$ so that:
    \begin{enumerate}[label=\textbf{(\arabic*)}]
        \item $z_0 y_1 y_2 \cdots y_{k - 1}$ is the path between $z_0$ and $y_{k - 1}$;
        \item $z_0 z_1 z_2$ is the pendent path of length two from $z_0$;
        \item $z_0 x_1 x_2 \cdots x_{\frac{n}{2} - 2}$ is the pendent path of length $\frac{n}{2} - 2$ from $z_0$;
        \item $y_{k - 1} w$ is the pendent path of length one from $y_{k - 1}$;
        \item $y_{k - 1} y_k y_{k + 1} \cdots y_{\frac{n}{2} - 2}$ is the pendent path of length $\frac{n}{2} - k - 1$ from $y_{k - 1}$;
    \end{enumerate}
    see Figure \ref{fig_3}. Moreover, let $\Tr(z_0) = \beta$ and note that $n = k^2 + k + 4$. By repeatedly using Lemma~\ref{neighbors_lemma}, we obtain $\Tr(z_1) = \beta + (k^2 + k)$, $\Tr(z_2) = \beta + (2k^2 + 2k + 2)$, $\Tr(x_i) = \beta + i(i + 3)$ for any $i \in \left[ \frac{n}{2} - 2 \right]$ and $\Tr(y_j) = \beta + j(j + 1)$ for any $j \in [k - 1]$, alongside $\Tr(w) = \beta + (2k^2 + 2)$ and $\Tr(y_j) = \beta + j(j + 3) - 2k + 2$ for any $j \in \{k, k + 1, \ldots, \frac{n}{2} - 2\}$. Since $\Tr(z_2) > \Tr(w) > \Tr(z_1)$, we have that $C_{n - 3}\left(\frac{n}{2} - 1, 2; \frac{n}{2} + k - 2, 1\right)$ is TI if and only if the following hold:
    \begin{enumerate}[label=\textbf{(\alph*)}]
        \item $\Tr(x_i)\neq \Tr(y_j)$ for any $i \in \left[ \frac{n}{2} - 2 \right]$ and $j \in \left[ \frac{n}{2} - 2 \right]$;
        \item none of the $\Tr(x_i)$ values are equal to $\Tr(z_1)$, $\Tr(z_2)$ or $\Tr(w)$, for any $i \in \left[ \frac{n}{2} - 2\right]$;
        \item none of the $\Tr(y_j)$ values are equal to $\Tr(z_1)$, $\Tr(z_2)$ or $\Tr(w)$, for any $j \in \left[ \frac{n}{2} - 2\right]$.
    \end{enumerate}

\begin{figure}[t]
\centering
\begin{tikzpicture}[scale=1.0]
\tikzstyle{vertex}=[draw,circle,minimum size=4pt,inner sep=1pt,fill=black]
\tikzstyle{edge}=[draw,thick]
\tikzstyle{dedge}=[draw,thick,dashed]

\node[vertex, label=below:$z_0$] (0) at (0, 0) {};
\node[vertex, label=left:$z_1$] (1) at (0, 1) {};
\node[vertex, label=right:$z_2$] (12) at (1, 1) {};

\node[vertex, label=below:$y_1$] (2) at (1, 0) {};
\node[vertex, label=below:$y_2$] (3) at (2, 0) {};
\node[vertex, label=below:$y_{k-2}$] (4) at (4, 0) {};
\node[vertex, label=below:$y_{k-1}$] (5) at (5, 0) {};
\node[vertex, label=right:$w$] (6) at (5, 1) {};
\node[vertex, label=below:$y_{k}$] (7) at (6, 0) {};
\node[vertex, label=below:$y_{\frac{n}{2} - 2}$] (8) at (8, 0) {};

\node[vertex, label=below:$x_1$] (9) at (-1, 0) {};
\node[vertex, label=below:$x_2$] (10) at (-2, 0) {};
\node[vertex, label=below:$x_{\frac{n}{2} - 2}$] (11) at (-4, 0) {};

\path[edge] (0) -- (1);
\path[edge] (0) -- (2);
\path[edge] (2) -- (3);
\path[dedge] (3) -- (4);
\path[edge] (4) -- (5);
\path[edge] (5) -- (6);
\path[edge] (5) -- (7);
\path[dedge] (7) -- (8);
\path[edge] (0) -- (9);
\path[edge] (9) -- (10);
\path[dedge] (10) -- (11);
\path[edge] (1) -- (12);
\end{tikzpicture}
\caption{The graph $C_{n - 3}\left(\frac{n}{2} - 1, 2; \frac{n}{2} + k - 2, 1\right)$ from Lemma \ref{even_sol_2}.}
\label{fig_3}
\end{figure}
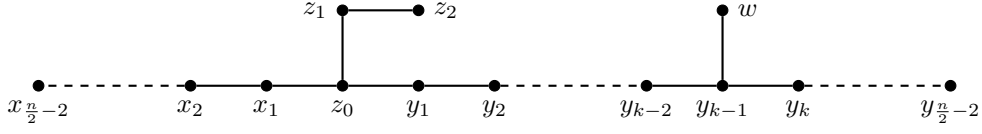

    If $\Tr(x_i) = \Tr(y_j)$ for some $i \in \left[ \frac{n}{2} - 2 \right]$ and $j \in [k - 1]$, it follows that $i(i + 3) = j(j + 1)$,  which implies that $j>i$ with $(j-i)(j+i+1)=2i$. Thus we get $2i\ge 2i+2$ as a clear contradiction.
     Now, suppose that $\Tr(x_i) = \Tr(y_j)$ is satisfied for some $i \in \left[ \frac{n}{2} - 2 \right]$ and $j \in \{k, k + 1, \ldots, \frac{n}{2} - 2\}$. In this case, we obtain $(j - i)(j + i + 3) = 2k - 2$. Since $i \ge 1$ and $j \ge k$, we have $j + i + 3 \ge k + 4$. If $j - i = 1$, then $(j - i)(j + i + 3) = 2j + 2 \ge 2k + 2$, while if $j - i \ge 2$, we get $(j - i)(j + i + 3) \ge 2(k + 4) = 2k + 8$. Either way, we reach a contradiction, which means that condition (a) always holds.

    If $\Tr(x_i) = \Tr(z_1)$ for some $i \in \left[ \frac{n}{2} - 2 \right]$, we get $i(i + 3) = k^2 + k$, which cannot hold for the same reason above. On the other hand, if $\Tr(x_i) = \Tr(z_2)$ for some $i \in \left[ \frac{n}{2} - 2 \right]$, it follows that $i(i + 3) = 2k^2 + 2k + 2 = 2n - 6$, thus implying $(2i + 3)^2 = 8n - 15$. However, $4n - 15$ and $8n - 15$ cannot be a perfect square at the same time. Indeed, if $4n - 15 = A^2$ and $8n - 15 = B^2$ hold for some $A, B \in \mathbb{Z}$, we get $B^2 - 2A^2 = 15$. Therefore, we have $A^2 + B^2 \equiv 0 \pmod{3}$, which is possible only if $A \equiv 0 \pmod{3}$ and $B \equiv 0 \pmod{3}$. This leads to a contradiction because $9 \nmid 15$. Now, suppose that $\Tr(x_i) = \Tr(w)$ for some $i \in \left[ \frac{n}{2} - 2 \right]$. In this case, we obtain $i(i + 3) = 2k^2 + 2$, which is equivalent to $(2i + 3)^2 = 8k^2 + 17$. Therefore, condition (b) is satisfied if and only if $8k^2 + 17\notin \mathcal{PS}$.

    We trivially observe that $\Tr(y_j) = \Tr(z_1)$ cannot be true for any $j \in \left[ \frac{n}{2} - 2\right]$, since $\Tr(y_{k - 1}) < \Tr(z_1) < \Tr(y_k)$. If we have $\Tr(y_j) = \Tr(z_2)$ for some $j \in \left[ \frac{n}{2} - 2\right]$, it follows that $j \ge k$, hence $j(j + 3) - 2k + 2 = 2k^2 + 2k + 2$, which is equivalent to $(2j + 3)^2 = 8k^2 + 16k + 9$. Similarly, if $\Tr(y_j) = \Tr(w)$ holds for some $j \in \left[ \frac{n}{2} - 2\right]$, we have $j \ge k$, which means that $j(j + 3) - 2k + 2 = 2k^2 + 2$, i.e., $(2j + 3)^2 = 8k^2 + 8k + 9$. With all of this in mind, we conclude that condition (c) holds if and only if $\{8k^2 + 8k + 9,8k^2 + 16k + 9\}\cap \mathcal{PS}=\emptyset$.
\end{proof}

The next lemma can be proved in an entirely analogous manner to Lemma \ref{even_sol_2}, so we omit its proof.

\begin{lemma}\label{even_sol_3}
    Let $n \ge 14$ be even and suppose that $4n - 15\in \mathcal{PS}$. Also, let $k = \frac{\sqrt{4n - 15} - 1}{2}$. Then $C_{n - 3}\left(\frac{n}{2} - 1, 2; \frac{n}{2} + k - 3, 1\right)$ is TI if and only if $\{ 8k^2 - 8k + 25, 8k^2 + 9, 8k^2 + 16k + 1 \} \cap \mathcal{PS} = \emptyset$.
\end{lemma}

\begin{lemma}\label{even_sol_4}
    For any even $n \ge 14$ such that $\{ 4n - 15, 4n - 7 \} \cap \mathcal{PS} = \emptyset$ and $8n - 23 \in \mathcal{PS}$, $C_{n - 3}\left(\frac{n}{2} - 1, 2; \frac{n}{2} + \frac{\sqrt{8n - 23} - 1}{2} - 2, 1\right)$ is TI.
\end{lemma}
\begin{proof}
\begin{figure}[t]
\centering
\begin{tikzpicture}[scale=1.0]
\tikzstyle{vertex}=[draw,circle,minimum size=4pt,inner sep=1pt,fill=black]
\tikzstyle{edge}=[draw,thick]
\tikzstyle{dedge}=[draw,thick,dashed]

\node[vertex, label=below:$z_0$] (0) at (0, 0) {};
\node[vertex, label=left:$z_1$] (1) at (0, 1) {};
\node[vertex, label=right:$z_2$] (12) at (1, 1) {};

\node[vertex, label=below:$y_1$] (2) at (1, 0) {};
\node[vertex, label=below:$y_2$] (3) at (2, 0) {};
\node[vertex, label=below:$y_{k-2}$] (4) at (4, 0) {};
\node[vertex, label=below:$y_{k-1}$] (5) at (5, 0) {};
\node[vertex, label=right:$w$] (6) at (5, 1) {};
\node[vertex, label=below:$y_{k}$] (7) at (6, 0) {};
\node[vertex, label=below:$y_{\frac{n}{2} - 2}$] (8) at (8, 0) {};

\node[vertex, label=below:$x_1$] (9) at (-1, 0) {};
\node[vertex, label=below:$x_2$] (10) at (-2, 0) {};
\node[vertex, label=below:$x_{\frac{n}{2} - 2}$] (11) at (-4, 0) {};

\path[edge] (0) -- (1);
\path[edge] (0) -- (2);
\path[edge] (2) -- (3);
\path[dedge] (3) -- (4);
\path[edge] (4) -- (5);
\path[edge] (5) -- (6);
\path[edge] (5) -- (7);
\path[dedge] (7) -- (8);
\path[edge] (0) -- (9);
\path[edge] (9) -- (10);
\path[dedge] (10) -- (11);
\path[edge] (1) -- (12);
\end{tikzpicture}
\caption{The graph $C_{n - 3}\left(\frac{n}{2} - 1, 2; \frac{n}{2} + \frac{\sqrt{8n - 23} - 1}{2} - 2, 1\right)$ from Lemma \ref{even_sol_4}.}
\label{fig_4}
\end{figure}
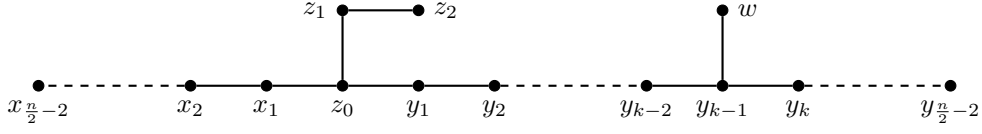

    Let $k = \frac{\sqrt{8n - 23} - 1}{2}$ and note that $n = \frac{k^2 + k + 6}{2}$. Let $z_0$ and $y_{k - 1}$ be the branching vertices of $C_{n - 3}\left(\frac{n}{2} - 1, 2; \frac{n}{2} + \frac{\sqrt{8n - 23} - 1}{2} - 2, 1\right)$ so that:
    \begin{enumerate}[label=\textbf{(\arabic*)}]
        \item $z_0 y_1 y_2 \cdots y_{k - 1}$ is the path between $z_0$ and $y_{k - 1}$;
        \item $z_0 z_1 z_2$ is the pendent path of length two from $z_0$;
        \item $z_0 x_1 x_2 \cdots x_{\frac{n}{2} - 2}$ is the pendent path of length $\frac{n}{2} - 2$ from $z_0$;
        \item $y_{k - 1} w$ is the pendent path of length one from $y_{k - 1}$;
        \item $y_{k - 1} y_k y_{k + 1} \cdots y_{\frac{n}{2} - 2}$ is the pendent path of length $\frac{n}{2} - k - 1$ from $y_{k - 1}$;
    \end{enumerate}
    see Figure \ref{fig_4}. Now, let $\Tr(z_0) = \beta$. By repeatedly using Lemma \ref{neighbors_lemma}, we have $\Tr(z_1) = \beta + \frac{k^2 + k - 2}{2}$, $\Tr(z_2) = \beta + (k^2 + k)$, $\Tr(x_i) = \beta + i(i + 3)$ for any $i \in \left[ \frac{n}{2} - 2 \right]$ and $\Tr(y_j) = \beta + j(j + 1)$ for any $j \in [k - 1]$, as well as $\Tr(w) = \beta + \frac{3k^2 - k + 2}{2}$ and $\Tr(y_j) = \beta + j(j + 3) - 2k + 2$ for any $j \in \{k, k + 1, \ldots, \frac{n}{2} - 2\}$. Since $\Tr(w) > \Tr(z_2) > \Tr(z_1)$, we conclude that $C_{n - 3}\left(\frac{n}{2} - 1, 2; \frac{n}{2} + k - 2, 1\right)$ is TI if and only if the following hold:
    \begin{enumerate}[label=\textbf{(\alph*)}]
        \item $\Tr(x_i)\neq \Tr(y_j)$ for any $i \in \left[ \frac{n}{2} - 2 \right]$ and $j \in \left[ \frac{n}{2} - 2 \right]$;
        \item none of the $\Tr(x_i)$ values are equal to $\Tr(z_1)$, $\Tr(z_2)$ or $\Tr(w)$, for any $i \in \left[ \frac{n}{2} - 2\right]$;
        \item none of the $\Tr(y_j)$ values are equal to $\Tr(z_1)$, $\Tr(z_2)$ or $\Tr(w)$, for any $j \in \left[ \frac{n}{2} - 2\right]$.
    \end{enumerate}
    We can show that condition (a) holds in the same way as in Lemma \ref{even_sol_2} and omit the proof here.

    Suppose that $\Tr(x_i) = \Tr(z_1)$ is true for some $i \in \left[ \frac{n}{2} - 2 \right]$. In this case, we get $i(i + 3) = \frac{k^2 + k - 2}{2} = n - 4$, i.e., $(2i + 3)^2 = 4n - 7$, which cannot be satisfied because $4n - 7$ is not a perfect square. Similarly, if we have $\Tr(x_i) = \Tr(z_2)$ for some $i \in \left[ \frac{n}{2} - 2 \right]$, we obtain $i(i + 3) = 2n - 6$, i.e., $(2i + 3)^2 = 8n - 15$, which leads to a contradiction because $8n - 23$ and $8n - 15$ cannot be a perfect square at the same time. Now, suppose that $\Tr(x_i) = \Tr(w)$ holds for some $i \in \left[ \frac{n}{2} - 2 \right]$. In this case, we obtain $i(i + 3) = \frac{3k^2 - k + 2}{2}$, which is equivalent to $(2i + 3)^2 = 6k^2 - 2k + 13$. Observe that $6k^2 - 2k + 13 \equiv -(k + 1)^2 \pmod{7}$. Since $0$ is the only quadratic residue modulo seven whose additive inverse is also a quadratic residue, we have $k = 7\ell - 1$ for some $\ell \in \mathbb{N}$. Therefore, $6k^2 - 2k + 13 = 49(6\ell^2 - 2\ell) + 21$, which yields a contradiction because $49 \nmid 21$. With all of this in mind, we conclude that condition (b) holds.

    Now, suppose that $\Tr(y_j) = \Tr(z_1)$ holds for some $j \in \left[ \frac{n}{2} - 2\right]$. Since $\Tr(y_k) > \Tr(z_1)$, we have $j \le k - 1$, hence $j(j + 1) = \frac{k^2 + k - 2}{2} = n - 4$, which is equivalent to $(2j + 1)^2 = 4n - 15$. Thus, we reach a contradiction because $4n - 15$ is not a perfect square. We trivially observe that $\Tr(y_j) = \Tr(z_2)$ is not satisfied for any $j \in \left[ \frac{n}{2} - 2\right]$, since $\Tr(y_k) > \Tr(z_2) > \Tr(y_{k - 1})$. If we suppose that $\Tr(y_j) = \Tr(w)$ for some $j \in \left[ \frac{n}{2} - 2\right]$, it follows that $j \ge k$, which gives $j(j + 3) - 2k + 2 = \frac{3k^2 - k + 2}{2}$, i.e., $(2j + 3)^2 = 6k^2 + 6k + 5$. This leads to a contradiction because $6k^2 + 6k + 5 \equiv 2 \pmod{3}$, while $2$ is not a quadratic residue modulo three. Therefore, condition (c) is satisfied.
\end{proof}

Based on Lemma \ref{branching lemma}, with some routine computations, we have the following results.
\begin{lemma}\label{even_wiener_lemma_1}
    For any even $n \ge 14$, we have
   \begin{enumerate}[label=\textbf{(\arabic*)}]
        \item $W\left(S\left(\tfrac{n}{2} - 1, \tfrac{n}{2} - 4, 4 \right)\right) = \frac{n^3-6n^2+59n-96}{6}$;
        \item $W\left(C_{n - 4} \left(\tfrac{n}{2} - 2, 3; \tfrac{n}{2} - 1, 1\right)\right) = \frac{n^3 - 6n^2 + 41n - 36}{6}$.
    \end{enumerate}
\end{lemma}

\begin{lemma}\label{even_wiener_lemma_2}
    For any even $n \ge 14$ such that $4n - 15\in \mathcal{PS}$, we have
    \begin{enumerate}[label=\textbf{(\arabic*)}]
        \item $W\left(C_{n - 3}\left(\tfrac{n}{2} - 1, 2; \tfrac{n}{2} + \tfrac{\sqrt{4n - 15} - 1}{2} - 2, 1\right)\right) = \frac{2n^3-9n^2+58n-102-6\sqrt{4n-15}}{12}$;
        \item $W\left(C_{n - 3}\left(\tfrac{n}{2} - 1, 2; \tfrac{n}{2} + \tfrac{\sqrt{4n - 15} - 1}{2} - 3, 1\right)\right) = \frac{2n^3-9n^2+58n-78-18\sqrt{4n-15}}{12}$;
        \item $W\left(C_{n - 4}\left(\tfrac{n}{2} - 1, 2; \tfrac{n}{2} + \tfrac{\sqrt{4n - 15} - 1}{2} - 2, 2\right)\right) = \frac{n^3-6n^2+47n-96}{6}$.
    \end{enumerate}
\end{lemma}

\begin{lemma}\label{even_wiener_lemma_6}
    Let $n \ge 14$ be even and suppose that $4n - 15\in \mathcal{PS}$. Let $k = \frac{\sqrt{4n - 15} - 1}{2}$ and suppose that $8k^2 + 17\in \mathcal{PS}$. Let $\ell = \frac{\sqrt{8k^2 + 17} - 3}{2}$. Then
    \[
        W\left(C_{n - 4}\left(\tfrac{n}{2} - 2, 2; \tfrac{n}{2} + k - 3, 1; \tfrac{n}{2} - \ell - 1, 1\right)\right) = \frac{n^3 - 6n^2 + 47n - 90 - 18k - 6 \ell}{6}.
    \]
\end{lemma}

\begin{lemma}\label{even_wiener_lemma_4}
    Let $n \ge 14$ be even and suppose that $4n - 15\in \mathcal{PS}$. Let $k = \frac{\sqrt{4n - 15} - 1}{2}$ and suppose that $8k^2 + 8k + 9\in \mathcal{PS}$. Let $\ell = \frac{\sqrt{8k^2 + 8k + 9} - 3}{2}$. Then
    \[
        W\left(C_{n - 4}\left(\tfrac{n}{2} - 1, 2; \tfrac{n}{2} + k - 2, 1; \tfrac{n}{2} + \ell - 2, 1\right)\right) = \frac{n^3 - 6n^2 + 47n - 102 - 6k - 6\ell}{6}.
    \]
\end{lemma}

\begin{lemma}\label{even_wiener_lemma_5}
    Let $n \ge 14$ be even and suppose that $4n - 15\in \mathcal{PS}$. Let $k = \frac{\sqrt{4n - 15} - 1}{2}$ and suppose that $8k^2 + 16k + 9\in \mathcal{PS}$. Let $\ell = \frac{\sqrt{8k^2 + 16k + 9} - 3}{2}$. Then
    \[
        W\left(C_{n - 4}\left(\tfrac{n}{2} - 1, 2; \tfrac{n}{2} + k - 2, 1; \tfrac{n}{2} + \ell - 2, 1\right)\right) = \frac{n^3 - 6n^2 + 47n - 102 + 6k - 6\ell}{6}.
    \]
\end{lemma}

\begin{lemma}\label{even_wiener_lemma_3}
    For any even $n \ge 14$ such that $8n - 23\in \mathcal{PS}$, we have
    \[
        W\left(C_{n - 3}\left(\tfrac{n}{2} - 1, 2; \tfrac{n}{2} + \tfrac{\sqrt{8n - 23} - 1}{2} - 2, 1\right)\right) = \frac{2n^3 - 9n^2 + 70n - 126 - 6 \sqrt{8n - 23}}{12} .
    \]
\end{lemma}

We will also need the following technical result.

\begin{lemma}\label{square_technical}
    For any even $n \ge 34$ such that $4n - 15\in\mathcal{PS}$, at least one of the trees $C_{n - 3}\left(\tfrac{n}{2} - 1, 2; \tfrac{n}{2} + \tfrac{\sqrt{4n - 15} - 1}{2} - 2, 1\right)$ and $C_{n - 3}\left(\tfrac{n}{2} - 1, 2; \tfrac{n}{2} + \tfrac{\sqrt{4n - 15} - 1}{2} - 3, 1\right)$ is TI.
\end{lemma}
\begin{proof}
    By way of contradiction, suppose that neither $C_{n - 3}\left(\tfrac{n}{2} - 1, 2; \tfrac{n}{2} + \tfrac{\sqrt{4n - 15} - 1}{2} - 2, 1\right)$ nor $C_{n - 3}\left(\tfrac{n}{2} - 1, 2; \tfrac{n}{2} + \tfrac{\sqrt{4n - 15} - 1}{2} - 3, 1\right)$ is TI. Also, let $k = \frac{\sqrt{4n - 15} - 1}{2}$ and note that $k \ge 5$. By Lemmas \ref{even_sol_2} and \ref{even_sol_3}, there exist $A, B \in \mathbb{N}$ such that $A^2 \in \{ 8k^2 + 17, 8k^2 + 8k + 9, 8k^2 + 16k + 9 \}$ and $B^2 \in \{ 8k^2 - 8k + 25, 8k^2 + 9, 8k^2 + 16k + 1 \}$. Observe that $A$ and $B$ are both odd and that
    \begin{equation}\label{even_aux_1}
        (A - B)(A + B) \in \{ 8, 8k, 8k - 8, -8k + 8, 16k, 16k - 16, -16k + 16, 24k - 16 \}.
    \end{equation}
    Therefore, $|(A-B)(A+B)| \le 24k - 16$.

    Since $A > 2k \sqrt{2}$ and $B > 2k\sqrt{2} - 2$, we have $A + B > 4k \sqrt{2} - 2$, which means that
    \[
        |A - B| < \frac{24k - 6}{4k \sqrt{2} - 2} = 3\sqrt{2} + \frac{6 \sqrt{2} - 6}{4k \sqrt{2} - 2} \le 3 \sqrt{2} + \frac{6 \sqrt{2} - 6}{12 \sqrt{2} - 2} < 5 .
    \]
    From here, we obtain $A - B \in \{ 2, -2, 4, -4 \}$, hence
    \[
        (A+ B)(A - B) \in \{ 4A - 4, -4A - 4, 8A - 16, -8A - 16\}.
    \]
    With this in mind, \eqref{even_aux_1} leads to
    \begin{align*}
        A \in \{ 3, k - 3, k + 1, k + 2, 2&k - 4, 2k - 3, 2k - 1,\\
        2&k, 2k + 1, 2k + 2, 3k, 4k - 5, 4k - 3, 4k + 1, 6k - 3\} .
    \end{align*}
    Since $2k \sqrt{2} < A < 2k \sqrt{2} + 4$, we conclude that $A \in \{ 3k, 4k - 5, 4k - 3\}$.

    If $A = 3k$, then we have $k^2 \in \{ 17, 8k + 9, 16k + 9 \}$, which is impossible unless $k = 9$. However, in this case, $\{ 8k^2 - 8k + 25, 8k^2 + 9, 8k^2 + 16k + 1 \} \cap \mathcal{PS} = \emptyset$. On the other hand, if $A = 4k - 5$, then we get $8k^2 \in \{40k - 8, 48k - 16, 56k - 16\}$, which cannot be satisfied. Finally, if $A = 4k - 3$, then we obtain $8k^2 = \{ 24k + 8, 32k, 40k \}$, which is impossible unless $k = 5$. However, $\{ 8k^2 - 8k + 25, 8k^2 + 9, 8k^2 + 16k + 1 \} \cap \mathcal{PS} = \emptyset$ when $k = 5$.
\end{proof}

We proceed with the next auxiliary lemma.

\begin{lemma}\label{even_arm_lemma}
    Let $n \ge 14$ be even and such that $n \notin \{16, 24\}$, and suppose that one of the following holds:
    \begin{enumerate}[label=\textbf{(\arabic*)}]
        \item $4n - 15\in \mathcal{PS}$;
        \item  $\{4n - 15,4n - 7\}\cap \mathcal{PS}=\emptyset$ and $8n - 23\in \mathcal{PS}$.
    \end{enumerate}
    Also, let $T$ be a TI tree of order $n$ that attains the maximum Wiener index, with $v \in V(T)$ having the minimum transmission. Then $d_T(v) = 3$ and there is a pendent path of length two attached to $v$, while the other two components of $T - v$ have $\frac{n}{2} - 1$ and $\frac{n}{2} - 2$ vertices, respectively.
\end{lemma}
\begin{proof}
    By Lemma \ref{min_deg_3_lemma}, we have $d_T(v) \ge 3$. Let $k = d_T(v)$ and let $T_1, T_2, \ldots, T_k$ be the components of $T - v$. Also, for each $i \in [k]$, let $a_i = |T_i|$. Due to Lemma \ref{compo}, we may assume without loss of generality that $\frac{n}{2} - 1 \ge a_1 > a_2 > \cdots > a_k \ge 1$.

    Note that $\sum_{i = 3}^k a_i = 1$ cannot hold, since $v$ has the smallest transmission in $T$. We aim to prove $\sum_{i = 3}^k a_i = 2$.   By way of contradiction, we suppose that $\sum_{i = 3}^k a_i \ge 3$. If $\sum_{i = 3}^k a_i \ge 4$, then the repeated use of Lemma \ref{arm_lemma} and Corollary \ref{majorization_lemma} gives $W(T) \le W\left(S\left(\frac{n}{2} - 1, \frac{n}{2} - 3, 4\right)\right)$. On the other hand, if $\sum_{i = 3}^k a_i = 3$, then we have $a_1 = \frac{n}{2} - 1$ and $a_2 = \frac{n}{2} - 3$. Let $u$ be the neighbor of $v$ from $T_1$. Observe that $d_T(u) \ge 3$, since otherwise, the neighbor of $u$ distinct from $v$ will have the same transmission as the neighbor of $v$ from $T_2$. Thus, using repeatedly Lemma \ref{arm_lemma} and Corollary~\ref{majorization_lemma}, we have $W(T) \le W\left(C_{n - 4} \left(\tfrac{n}{2} - 2, 3; \tfrac{n}{2} - 1, 1\right)\right)$. Combining the above argument with Lemma \ref{even_wiener_lemma_1}, we obtain
    \begin{equation}\label{even_aux_2}
        W(T) \le \frac{n^3 - 6n^2 + 59n - 96}{6}.
    \end{equation}
    To finalize the proof, it suffices to reach a contradiction. We now divide into the following cases.

    \bigskip\noindent
    \textbf{Case 1.} $4n - 15\in \mathcal{PS}$.

 Since $n \notin \{16, 24\}$, observe that $n \ge 34$. Therefore, Lemma \ref{square_technical} implies that at least one of the trees $C_{n - 3}\left(\tfrac{n}{2} - 1, 2; \tfrac{n}{2} + \tfrac{\sqrt{4n - 15} - 1}{2} - 2, 1\right)$ and $C_{n - 3}\left(\tfrac{n}{2} - 1, 2; \tfrac{n}{2} + \tfrac{\sqrt{4n - 15} - 1}{2} - 3, 1\right)$ is TI. Note that $T$ is a TI tree with maximum Wiener index. By Lemma \ref{even_wiener_lemma_2}, we have
    \begin{equation}\label{even_aux_3}
        W(T) \ge \frac{2n^3-9n^2+58n-78-18\sqrt{4n-15}}{12} .
    \end{equation}
    By combining \eqref{even_aux_2} and \eqref{even_aux_3}, we reach
    \[
        3n^2 - 60n + 114 - 18 \sqrt{4n - 15} \le 0, \qquad \mbox{i.e.,} \qquad n^2 - 20n + 38 - 6 \sqrt{4n - 15} \le 0 ,
    \]
    which clearly cannot hold when $n \ge 34$.

    \bigskip\noindent
    \textbf{Case 2.} $\{4n - 15,4n - 7\}\cap \mathcal{PS}=\emptyset$ and $8n - 23\in \mathcal{PS}$.

    In this case, we have $n \ge 18$. Lemma \ref{even_sol_4} implies that $C_{n - 3}\left(\frac{n}{2} - 1, 2; \frac{n}{2} + \frac{\sqrt{8n - 23} - 1}{2} - 2, 1\right)$ is TI. From Lemma~\ref{even_wiener_lemma_3} and the fact that $T$ is a TI tree with maximum Wiener index, we obtain
    \begin{equation}\label{even_aux_4}
        W(T) \ge \frac{2n^3 - 9n^2 + 70n - 126 - 6 \sqrt{8n - 23}}{12} .
    \end{equation}
    By combining \eqref{even_aux_2} and \eqref{even_aux_4}, we get
    \[
        3n^2 - 48n + 66 - 6 \sqrt{8n - 23} \le 0, \qquad \mbox{i.e.,} \qquad n^2 - 16n + 22 - 2\sqrt{8n - 23} \le 0,
    \]
    which obviously cannot hold when $n \ge 18$.
\end{proof}

Observe that if $n \ge 14$ and $4n - 15 \in \mathcal{PS}$, then we have $n \in \{16, 24\}$ or $n \ge 34$. With this in mind, we settle cases (ii) and (iii) of Theorem \ref{even_th} as follows.

\begin{proposition}\label{even_th_case_2}
Let $n \ge 34$ be even with $4n - 15\in \mathcal{PS}$. Suppose that $k = \frac{\sqrt{4n - 15} - 1}{2}$ with $\{8k^2 + 17, 8k^2 + 8k + 9, 8k^2 + 16k + 9\}\cap\mathcal{PS}=\emptyset$. Then $C_{n - 3}\left(\tfrac{n}{2} - 1, 2; \frac{n}{2} + \tfrac{\sqrt{4n - 15} - 1}{2} - 2, 1\right)$ is the unique TI tree of order $n$ that attains the maximum Wiener index.
\end{proposition}
\begin{proof}
    Let $T$ be a TI tree of order $n$ with the maximum Wiener index such that  $v \in V(T)$ has the minimum transmission. Lemma \ref{even_arm_lemma} implies that $d_T(v) = 3$ and $T - v = \bigcup_{i = 1}^3 T_i$, where $|T_1| = \frac{n}{2} - 1$, $|T_2| = \frac{n}{2} - 2$ and $|T_3| = 2$. Let $v u_1 u_2$ be the pendent path of length two attached to $v$. If $T_1$ is a path from $v$, then the vertex on this path with distance $k$ to $v$ has the same transmission in $T$ as $u_1$, contradicting the fact that $T$ is a TI tree. Therefore $T_1$ must contain at least one branching vertex in $T$. We denote by $w$ the branching vertex in $T$ from $T_1$ closest to $v$. A similar conclusion can be made if $d_T(w, v) \ge k$. Therefore, we have $d_T(w, v) \le k - 1$.

    By repeatedly applying Lemma \ref{arm_lemma} and Corollary \ref{majorization_lemma} in the same way as in Lemma \ref{odd_type_a_1}, we conclude that $W(T) \le W(T')$, where $T' \cong C_{n - 3}\left(\frac{n}{2} - 1, 2; \frac{n}{2} + t - 1, 1 \right)$ and $t = d_{T}(w, v)$, with equality holding if and only if $T \cong T'$. Since $t \le k - 1$, Lemma \ref{wiener_fusion_lemma} implies that
    \[
        W(T') \le W\left( C_{n - 3}\left(\tfrac{n}{2} - 1, 2; \tfrac{n}{2} + k - 2, 1\right) \right),
    \]
    with equality holding if and only if $t = k - 1$. With all of this in mind, we conclude that $W(T) < W\left( C_{n - 3}\left(\tfrac{n}{2} - 1, 2; \tfrac{n}{2} + k - 2, 1\right) \right)$ holds if $T\ncong C_{n - 3}\left(\tfrac{n}{2} - 1, 2; \tfrac{n}{2} + k - 2, 1\right)$. The result now follows from Lemma \ref{even_sol_2}.
\end{proof}

\begin{proposition}\label{even_th_case_3}
Let $n \ge 34$ be even and suppose that $4n - 15\in \mathcal{PS}$. Let $k = \frac{\sqrt{4n - 15} - 1}{2}$ and $\{8k^2 + 17, 8k^2 + 8k + 9, 8k^2 + 16k + 9\}\cap\mathcal{PS}\neq\emptyset$. Then $C_{n - 3}\left( \tfrac{n}{2} - 1, 2; \tfrac{n}{2} + \tfrac{\sqrt{4n - 15} - 1}{2} - 3, 1\right)$ is the unique TI tree of order $n$ that attains the maximum Wiener index.
\end{proposition}
\begin{proof}
    Note that Lemmas \ref{even_sol_2} and \ref{square_technical} imply that $C_{n - 3}\left( \tfrac{n}{2} - 1, 2; \tfrac{n}{2} + k - 3, 1\right)$ is TI. Now, let $T$ be a TI tree of order $n$ with the maximum Wiener index such that $v \in V(T)$ has the minimum transmission. From Lemma \ref{even_arm_lemma}, we get $d_T(v) = 3$ and $T - v = \bigcup_{i = 1}^3 T_i$, where $|T_1| = \frac{n}{2} - 1$, $|T_2| = \frac{n}{2} - 2$ and $|T_3| = 2$. Analogously to Proposition \ref{even_th_case_2}, we conclude that there exists a branching vertex $w$ in $T$ from $T_1$ closest to $v$ with $d_T(w, v) \le k - 1$.

    If $d_T(w, v) \le k - 2$, then it can be shown analogously to Proposition \ref{even_th_case_2} that
    \[
        W(T) \le W\left( C_{n - 3}\left( \tfrac{n}{2} - 1, 2; \tfrac{n}{2} + k - 3, 1\right) \right) ,
    \]
    with equality holding only if $T \cong C_{n - 3}\left( \tfrac{n}{2} - 1, 2; \tfrac{n}{2} + k - 3, 1\right)$. From the TI property of the tree $C_{n - 3}\big(\tfrac{n}{2}-1, 2; \tfrac{n}{2} + k - 3, 1\big)$, we get $T \cong C_{n - 3}\left( \tfrac{n}{2} - 1, 2; \tfrac{n}{2} + k - 3, 1\right)$.

    Now, suppose that $d_T(w, v) = k - 1$. If $d_T(w) \ge 4$ or there is no leaf attached to $w$, then the repeated use of Lemma \ref{arm_lemma} and Corollary \ref{majorization_lemma} gives
    \[
        W(T) \le W\left(C_{n - 4}\left(\tfrac{n}{2} - 1, 2; \tfrac{n}{2} + k - 2, 2\right)\right) .
    \]
    From Lemma \ref{even_wiener_lemma_2}, we obtain
    \[
        W(T) \le \frac{n^3-6n^2+47n-96}{6} .
    \]
    A routine computation shows that $W\left(C_{n - 3}\left(\tfrac{n}{2} - 1, 2; \tfrac{n}{2} + k - 3, 1\right)\right) > \frac{n^3-6n^2+47n-96}{6}$, thus yielding a contradiction. This means that $d_T(w) = 3$ and there is a leaf attached to $w$. Let $w'$ be the leaf attached to $w$ in $T$ and let $T_4$ be the component of $T - w$ that contains neither $v$ nor $w'$. We finalize the proof by splitting the problem into three corresponding cases.

    \bigskip\noindent
    \textbf{Case 1:} $8k^2 + 17\in \mathcal{PS}$.

    In this case, following the proof of Lemma \ref{even_sol_2}, we conclude that if $T_2$ is a path from $v$, then the vertex on this path with distance $\ell = \frac{\sqrt{8k^2 + 17} - 3}{2}$ to $v$ has the same transmission in $T$ as $w'$. A similar conclusion can be made if the branching vertex in $T$ from $T_2$ closest to $v$ has a distance to $v$ of at least $\ell$. Therefore, the repeated use of Lemma \ref{arm_lemma} and Corollary \ref{majorization_lemma}, together with Lemma \ref{wiener_fusion_lemma}, yields
    \[
        W(T) \le W\left(C_{n - 4}\left(\tfrac{n}{2} - 2, 2; \tfrac{n}{2} + k - 3, 1; \tfrac{n}{2} - \ell - 1, 1 \right)\right) .
    \]
    From Lemma \ref{even_wiener_lemma_6} and Lemma \ref{even_wiener_lemma_2} (2), we obtain
    \[
        W(T) \le \frac{n^3 - 6n^2 + 47n - 90 - 18k - 6 \ell}{6}<W\left(C_{n - 3}\left(\tfrac{n}{2} - 1, 2; \tfrac{n}{2} + k - 3, 1\right)\right)
    \]
   as a contradiction to the maximality of $T$.

    \bigskip\noindent
    \textbf{Case 2:} $8k^2 + 8k + 9\in \mathcal{PS}$.

    In this case, by the TI property of $T$, there exists a branching vertex in $T$ from $T_4$ closest to $w$ with distance to $v$ less than $\ell = \frac{\sqrt{8k^2 + 8k + 9} -  3}{2}$. By repeated use of Lemma \ref{arm_lemma} and Corollary~\ref{majorization_lemma}, alongside Lemma \ref{wiener_fusion_lemma}, we conclude that
    \[
        W(T) \le W\left(C_{n - 4}\left(\tfrac{n}{2} - 1, 2; \tfrac{n}{2} + k - 2, 1; \tfrac{n}{2} + \ell - 2, 1 \right)\right) .
    \]
    Lemma \ref{even_wiener_lemma_4} and Lemma \ref{even_wiener_lemma_2} (2) imply that
    \[
        W(T) \le \frac{n^3 - 6n^2 + 47n - 102 - 6k - 6\ell}{6}<W\left(C_{n - 3}\left(\tfrac{n}{2} - 1, 2; \tfrac{n}{2} + k - 3, 1\right)\right),
    \]
    which again yields a contradiction to the maximality of $T$.

    \bigskip\noindent
    \textbf{Case 3:} $8k^2 + 16k + 9\in \mathcal{PS}$.

    By using Lemma \ref{even_wiener_lemma_5}, this case can be resolved analogously to Case 2.
\end{proof}

Finally, we resolve case (iv) of Theorem \ref{even_th} through the next proposition, which can be proved via Lemma \ref{even_sol_4} analogously to Proposition \ref{even_th_case_2}.

\begin{proposition}\label{even_th_case_8}
For each even $n \ge 14$ such that $\{4n - 15,4n - 7\}\cap\mathcal{PS}=\emptyset$ and $8n - 23\in \mathcal{PS}$, $C_{n - 3}\left( \tfrac{n}{2} - 1, 2; \tfrac{n}{2} + \tfrac{\sqrt{8n - 23} - 1}{2} - 2, 1\right)$ is the unique TI tree of order $n$ that attains the maximum Wiener index.
\end{proposition}

Theorem \ref{even_th} now follows from Propositions \ref{even_th_case_1}, \ref{even_th_case_2}, \ref{even_th_case_3} and \ref{even_th_case_8}, as well as the computational results obtained by the \texttt{SageMath} script given in Appendix \ref{sage_script}.

\section{Some open problems}\label{sc_conclusion}

\ \indent In this paper we completely solve the Wiener index maximization problem of all TI trees of odd order $n\ge  7$. The maximum  TI trees of even order $n$ with respect to the Wiener index are also characterized for all $n\ge 14$ except for $n \ge 30$ with $4n - 7 \in \mathcal{PS}$, or $\{4n - 15, 4n - 7\} \cap \mathcal{PS} = \emptyset$ with $8n - 15 \in \mathcal{PS}$. In view of the above unsolved cases,  we would like to pose the following problem.

\begin{problem}
    Determine the maximum TI trees with respect to the Wiener index of even order $n \ge 30$ with $4n - 7 \in \mathcal{PS}$, and $\{4n - 15, 4n - 7\} \cap \mathcal{PS} = \emptyset$ with $8n - 15 \in \mathcal{PS}$, respectively.
\end{problem}

Considering the Wiener index maximization problem of the trees mentioned above, a natural problem occurs into our mind as follows.

\begin{problem}
    Characterize the minimum TI trees of a given order with respect to the Wiener index.
\end{problem}

As a class of graphs with a simple structure, trees are frequently studied in pure and applied graph theory. The Wiener index extremal problem can be extended in the following.

\begin{problem}
    Determine the extremal TI graphs with respect to the Wiener index among all general cycle-containing graphs under some structural constraints.
\end{problem}

Note that the \textit{Wiener complexity} \cite{AlKl2018} of a connected graph $G$ is just $W_c(G)=|\Tr(G)|$, that is, the number of distinct transmissions in $G$. For enlarging the class of TI graphs in this topic, we would like to end this paper with the following problem.

\begin{problem}
    Establish the inner relationship between the Wiener index and the Wiener complexity of a graph, such as providing a function of the Wiener index in terms of the Wiener complexity.
\end{problem}

\subsection*{Declarations}

\ \indent
\textbf{Data Availability} This work has no associated data.

\ \textbf{Author Contributions} All authors have made equal contributions to this work.

\ \textbf{Competing Interests} The authors have no competing interests.

\section*{Acknowledgements}

The authors sincerely thank two anonymous referees for their careful reading and some helpful
comments on our paper.
 I.\ Damnjanovi\'{c} is supported by the Ministry of Science, Technological Development and Innovation of the Republic of Serbia, grant number 451-03-137/2025-03/200102, and the Science Fund of the Republic of Serbia, grant \#6767, Lazy walk counts and spectral radius of threshold graphs --- LZWK. K.\ Xu is supported by NNSF of China (Grant No. 12271251).

\appendix

\section{\texorpdfstring{\texttt{SageMath}}{SageMath} script for TI trees of small order}\label{sage_script}

In this appendix section, we either establish that there is no TI tree of order $n$ or find all the TI trees of order $n$ attaining the maximum Wiener index, for each $n \in \{2, 3, \ldots, 24\}$. We accomplish this through a two-step computer-assisted search. In the first step, we use the efficient TI tree generation algorithm from \cite{StoDam2025} to find all the TI trees of order at most $24$. These trees are stored to the file \texttt{ti\_trees\_24.s6} in the \texttt{sparse6} format.

In the second step, we run a \texttt{SageMath} script that performs a brute-force search over all the obtained TI trees. The script shows that there is no TI tree of order $n \in \{2, 3, 4, 5, 6, 8, 10, 12\}$ and that for each remaining $n$, there is a unique TI tree of order $n$ with the maximum Wiener index. These extremal trees are, of course, precisely those given in Theorem \ref{odd_th}, Proposition \ref{even_secondary} and Theorem \ref{even_th}. The \texttt{SageMath} script code is given below.

\begin{lstlisting}[language = Python, frame = trBL, escapeinside={(*@}{@*)}, aboveskip=10pt, belowskip=10pt, numbers=left, rulecolor=\color{black}]
import numpy as np
from sage.all import *


def load_ti_trees(input_path):
    loaded_trees = {}

    tree_codes = []
    with open(input_path, "r") as opened_file:
        for line in opened_file:
            tree_codes.append(line.strip())

    for code in tree_codes:
        tree = Graph(code, sparse=True)

        n = tree.order()
        if n not in loaded_trees:
            loaded_trees[n] = []

        loaded_trees[n].append(tree)

    return loaded_trees


def find_maximum(n, ti_trees):
    solution_trees = []
    solution_wiener = -1

    for tree in ti_trees:
        wiener = tree.wiener_index()

        if wiener > solution_wiener:
            solution_trees = [tree]
            solution_wiener = wiener
        elif wiener == solution_wiener:
            solution_trees.append(tree)

    with open(f"solutions/{n:02}.txt", "w") as opened_file:
        counter = 0
        for tree in solution_trees:
            p = tree.plot()
            p.save(f"solutions/{n:02}_{counter:02}.png")
            opened_file.write(tree.sparse6_string())

            counter += 1


if __name__ == "__main__":
    ti_trees_24 = load_ti_trees("ti_trees_24.s6")

    for n in range(2, 25):
        if n not in ti_trees_24:
            continue

        find_maximum(n, ti_trees_24[n])
\end{lstlisting}

\end{document}